\documentclass[11pt]{article}
\usepackage{mathrsfs}
\usepackage{amsfonts}
\usepackage{}
\usepackage{amsmath,amssymb,amsthm,latexsym,amstext}
\usepackage[mathscr]{eucal}
\usepackage{rotate,graphics,epsfig,epstopdf}
\usepackage{float}
\usepackage{color}
\usepackage{subfigure}
\usepackage{indentfirst}
\usepackage{bm}

\newcommand{\be}{\begin{eqnarray}}
\newcommand{\ee}{\end{eqnarray}}
\newcommand{\by}{\begin{eqnarray*}}
\newcommand{\ey}{\end{eqnarray*}}
\newcommand{\bn}{\begin{enumerate}}
\newcommand{\en}{\end{enumerate}}
\newcommand{\ei}{\end{itemize}}

\newtheorem{theorem}{Theorem}[section]
\newtheorem{lemma}[theorem]{Lemma}

\newtheorem{remark}[theorem]{Remark}

\newtheorem{proposition}[theorem]{Proposition}

\renewcommand{\theequation}{\arabic{section}.\arabic{equation}}

\numberwithin{equation}{section}

\begin{document}
\date{}
\title{\bf The Smoluchowski-Kramers approximation with distribution-dependent potential and highly oscillating force  \footnote{This work was supported by
the Natural Science Foundation of Jiangsu Province, BK20230899 and
the National Natural Science Foundation of China, 11771207.
}}
\author{ Chungang Shi$^{1}$\footnote{Corresponding author, shichungang@njust.edu.cn } \hskip1cm Wei Wang$^{2}$\footnote{wangweinju@nju.edu.cn} \\
\texttt{{\scriptsize $^{1}$School of Mathematics and Statistics, Nanjing University of Science and Technology,
Nanjing, 210094, P. R. China}}\\
\texttt{{\scriptsize $^{2}$Department of Mathematics, Nanjing University,
Nanjing, 210023, P. R. China}}}\maketitle
\begin{abstract}
An approximation is derived for a Langevin equation with distribution-dependent potential and state-dependent, randomly fast oscillation.  By some estimates and a diffusion approximation the limiting equation is shown to be distribution-dependent stochastic differential equation (SDEs) driven by white noise.
\end{abstract}



\section{Introduction}\label{sec:intro}
  \setcounter{equation}{0}
  \renewcommand{\theequation}
{1.\arabic{equation}}
The Langevin equation describes the motion of a small particle in some media with fluctuation~\cite{RZ},
\begin{equation}\label{a1}
m\ddot{x}=-\nabla V(x)-\alpha\dot{x}+\sqrt{2\alpha\beta^{-1}}\dot{W}(t).
\end{equation}
Here $m$ is the mass of the particle, $\dot{W}(t)$ is the white noise describing the fluctuation and  $\beta^{-1}=k_{B}T$, $k_{B}$ denotes Boltzmann's constant and $T$ the absolute temperture.  This is Newton's equation of motion with two additional terms, a linear dissipation term $\alpha\dot{x}$ and stochastic force $\xi(t)=\sqrt{2\alpha\beta^{-1}}\dot{W}(t)$. They are related through fluctuation-dissipation relationship~\cite{K}
\begin{equation*}
\mathbb{E}(\xi(t)\xi(s))=\alpha\beta^{-1}\delta(t-s).
\end{equation*} 

When the particle's mass goes to zero in equation~(\ref{a1}), then the limit is called as the Smoluchowski-Kramers approximation.
There are lots of works on the Smoluchowski-Kramers approximation. For the Gaussian noise case, Hottovy~\cite{HMVW}  considered the Smoluchowski-Kramers limit of stochastic differential equations with arbitrary state-dependent friction by means of a theory of convergence of stochastic integrals. Cerrai et al.~\cite{CX,CX1} studied a Smoluchowski-Kramers approximation for an infinite dimensional system with state-dependent damping. Zine~\cite{Zi} investigated a Smoluchowski-Kramers approximation for a singular stochastic wave equations in two dimensions. For the highly oscillating random force case,
Wang et al.~\cite{WRDH} considered a nonlinear stochastic heat equation with randomly fast oscillation on the body force and on the boundary. By applying diffusion approximation, the ensemble averaged model which is a stochastic partial differential equation is derived with a homogenous boundary condition. Lv and Wang~\cite{LW} considered both nonlinear heat equation and singularly perturbed nonlinear wave equation with highly oscillating random force on the boundary and strong interaction and the limit in the sense of distribution is derived by using a diffusion approximation. Very recently, Lv and Wang~\cite{LW1} investigated a Smoluchowski-Kramers approximation with state dependent damping and highly random oscillation. For more related results on approximation we refer to~\cite{SW,F,CF1,CF3}.

Based on the above works, if we think about the interaction of the surrounding medium on the particle, in the paper, then we  consider the following stochastic differential equation with distribution-dependent potential and randomly fast oscillation, 
\begin{eqnarray}\label{equ:main}
\epsilon\ddot{x}^{\epsilon}(t)+\alpha\dot{x}^{\epsilon}(t)&=&-\nabla V(x^{\epsilon}(t),\mu^{\epsilon}_{t})+\frac{1}{\sqrt{\epsilon}}\int_{\mathbb{R}^{d}}\eta^{\epsilon}(t,x)\mu_{t}^{\epsilon}(dx)\\
x^{\epsilon}(0)&=&x_{0},\quad\dot{x}^{\epsilon}(0)=y_{0}\in \mathbb{R}^{d},\nonumber
\end{eqnarray}
where $\epsilon>0$ denotes the mass of the particle, $\alpha>0$ is the friction coefficient, $x^{\epsilon}(t)$ is the position of the particle at time $t$, $\mu_{.}^{\epsilon}=\mathcal{L}(x^{\epsilon}(\cdot))$ is the distribution of $x^{\epsilon}(\cdot)$. The external force $\eta^{\epsilon}(t,\cdot)=\eta(\frac{t}{\epsilon},\cdot)$ is highly oscillating random force satisfying some mixing property in time detailed in next section and $\int_{\mathbb{R}^{d}}\eta^{\epsilon}(t,x)\mu_{t}^{\epsilon}(dx)$ represents the fluctuating force of the surrounding medium on the particle. $\nabla V(\cdot): \mathbb{R}^{d}\to\mathbb{R}^{d}$ is a deterministic force which is assumed to be Lipschitz continuous function in two variables. 

Our goal is to obtain the limit of $x^{\epsilon}(t)$ as $\epsilon\to0$ which is viewed as some kind of Smoluchowski-Kramers approximation which describes the motion of a particle with small mass. In order to derive the approximation equation that is the limit of $\{x^{\epsilon}(t)\}_{0< \epsilon\leq 1}$, there are two key steps when we apply diffusion approximation to obtain the approximation equation. On the one hand, we need to show the tightness of solution $\{x^{\epsilon}(t)\}_{0< \epsilon\leq 1}$ of the equation (\ref{equ:main}) in space $C(0,T;\mathbb{R}^{d})$ for any $T>0$. In order to obtain some uniformly bounded estimations on the solution, considering the dependence of the fluctuating force on the distribution of the particle, we have to make some stronger assumptions about the fluctuating force specified later which is different from the previous works~\cite{LW, LW1, SW} and several additional efforts have to be done.
Because of the existence of singular term $\frac{1}{\sqrt{\epsilon}}\int_{\mathbb{R}^{d}}\eta^{\epsilon}(t,x)\mu_{t}^{\epsilon}(dx)$, we introduce a martingale to overcome the difficulty and to derive some bounded estimates and continuity of the solution $x^{\epsilon}$. On the other hand, we adopt the method of diffusion approximation to take the limit. We need to construct another martingale and pass the limit $\epsilon\to0$ to show that the limit solves a martingale problem which is equivalent to solving a stochastic differential equation. Then the stochastic differential equation is the approximation equation in the sense of distribution.

The rest of the paper is organized as follows. In section \ref{Pre;Res}, we give some notations, assumptions and our main result. Section \ref{Tight} is devoted to the proof of the tightness of the solution. The last section takes the limit of the solution by diffusion approximation.

\section{Preliminary and main result}\label{Pre;Res}
  \setcounter{equation}{0}
  \renewcommand{\theequation}
{2.\arabic{equation}}

Let $\|\cdot\|$ and $\langle\cdot,\cdot\rangle$ denote the norm and the product of $\mathbb{R}^{d}$ respectively. Let $(\Omega,\mathcal{F},\mathbb{P})$ be a complete probability space and $\mathbb{E}$ is the expectation with respect to $\mathbb{P}$. In addition, let $(\tilde{\Omega},\tilde{\mathcal{F}},\tilde{\mathbb{P}})$ be a version of $(\Omega,\mathcal{F},\mathbb{P})$, $\tilde{\mathbb{E}}$ is the expectation with respect to $\tilde{\mathbb{P}}$. Let $\mathcal{P}_{2}(\mathbb{R}^{d})$ denote the space of Borel probability measures on $\mathbb{R}^{d}$ with 
\begin{equation*}
\int_{\mathbb{R}^{d}}\|x\|^{2}\mu(dx)<\infty
\end{equation*} 
for every $\mu\in\mathcal{P}_{2}$. For any $\mu, \nu$ in $\mathcal{P}_{2}(\mathbb{R}^{d})$, define the following Monge-Kantorovich (or Wasserstein) distance:
\begin{equation}\label{dMK}
\mathcal{W}_{2}(\mu,\nu)=\inf_{\pi\in \Pi(\mu,\nu)}\Big[\int\int_{\mathbb{R}^{d}\times\mathbb{R}^{d}}\|x-y\|^{2}\Pi(dx,dy)\Big]^{\frac{1}{2}},
\end{equation}
where $\Pi(\mu,\nu)$ is the set of Borel probability measures $\pi$ on $\mathbb{R}^{d}\times\mathbb{R}^{d}$ with first and second marginals $\mu$ and $\nu$. Equivalently, for $\mu,\nu\in\mathcal{P}_{2}$,
\begin{equation}\label{rvMK}
\mathcal{W}_{2}(\mu,\nu)=\inf_{(X,\bar{X})}[\mathbb{E}\|X-\bar{X}\|^{2}]^{\frac{1}{2}}
\end{equation}
with random variables $X$ and $\bar{X}$ in $\mathbb{R}^{d}$ having laws $\mu$ and $\nu$, respectively.

Next, we make the following assumptions on the coefficients and noise.

$(\mathbf{H_{1}}).$ There exists a constant $L_{V}>0$ such that
\begin{equation}
\|\nabla V(x,\mu)-\nabla V(y,\nu)\|\leq L_{V}(\|x-y\|+\mathcal{W}_{2}(\mu,\nu)),
\end{equation}
for any $\mu,\nu\in\mathcal{P}_{2}$ and $x,y\in\mathbb{R}^{d}$.

$(\mathbf{H_{2}}).$ 
 \begin{equation*}\label{h1}
\int_{0}^{\infty}m(t)dt\triangleq K<\infty,
\end{equation*}
where
\begin{equation*}
m(t)=\sup_{s\geq0}\sup_{\substack{A\in\mathcal{F}_{0}^{s}\\B\in\mathcal{F}_{s+t}^{\infty}}}|\mathbb{P}(AB)-\mathbb{P}(A)\mathbb{P}(B)|
\end{equation*}
and
\begin{equation*}
\mathcal{F}_{s}^{t}=\sigma\{\eta(\tau,u): s\leq\tau\leq t\}
\end{equation*}
for any function $u\in\mathbb{R}^{d}$.
Moreover, 
\begin{eqnarray*}
\int_{0}^{\infty}m(s)^{\frac{1}{4}}ds<\infty.
\end{eqnarray*}

$(\mathbf{H_{3}})$.  
The stationary process $\eta(t,u)$ is continuously differentiable with respect to the first variable and is twice continuously differentiable with respect to the second variable a.s. and satisfies 
$\|\tilde{\mathbb{E}}\partial_{t}^{k}\partial_{x}^{l}\eta(t,u)\|\leq M_{\eta}, k=0,1,l \in \mathbb{N}, 0\leq k+l\leq2$, $\tilde{\mathbb{E}}$ is the expectation with respect to the distribution of $u$.

Then we give the main result of the paper.
\begin{theorem}\label{equ:thm1}
Assume $(\mathbf{H_{1}})$-$(\mathbf{H_{3}})$ hold, for any $T>0$, the solution $x^{\epsilon}$ of equation (\ref{equ:main}) converges in distribution as $\epsilon\to0$ to $x(t)$ in space $C(0,T;\mathbb{R}^{d})$ with $x$ satisfying
\begin{eqnarray}
dx(t)=-\frac{1}{\alpha}\nabla V(x(s),\mu_{s})ds+\frac{\sqrt{\Sigma}}{\alpha\sqrt{\beta}}dB(t),
\end{eqnarray}
where $\Sigma=\mathbb{E}(\tilde{\mathbb{E}}\eta^{\epsilon}(t,x^{\epsilon}(t))\otimes\tilde{\mathbb{E}}\eta^{\epsilon}(t,x^{\epsilon}(t)))$, $\beta=-m'(0)\geq0$ and $B(t)$ is a $\mathbb{R}^{d}$-value standard Brownian motion.
\end{theorem}
The following martingale result may be used frequently. Let $\mathcal{X}$ be the space of all $\mathcal{F}_{0}^{t}-$progressively measurable process $X$ such that
\begin{eqnarray*}
\sup_{0\leq t\leq T}\mathbb{E}|X(t)|<\infty.
\end{eqnarray*}
Let $\mathcal{N}=\{X\in\mathcal{X}:\sup_{0\leq t\leq T}\mathbb{E}|X(t)|=0\},$ then the quotient space $\mathcal{X}/ \mathcal{N}$ is a Banach space with norm
\begin{eqnarray*}
\|X\|=\sup_{0\leq t\leq T}\mathbb{E}|X(t)|.
\end{eqnarray*} 
Then we have following result
\begin{proposition}\label{proa1}~{\rm (\cite{EK})}
Let $Y\in\mathcal{X}$. If 
\begin{equation*}
\{s^{-1}\mathbb{E}[Y(t+s)-Y(t)\big|\mathcal{F}_{0}^{t}]:s>0, t\geq0\}
\end{equation*}
is uniformly integrable and in probability
\begin{equation*}
s^{-1}\mathbb{E}[Y(t+s)-Y(t)\big|\mathcal{F}_{0}^{t}]\to Z(t), as\,s\,\to0^{+},a.e. \quad t,
\end{equation*}
then 
\begin{equation*}
Y(t)-\int_{0}^{t}Z(s)ds
\end{equation*}
is an $\mathcal{F}_{0}^{t}$-martingale.
\end{proposition}

\begin{remark}
By the definition of expectation, the oscillating term in (\ref{equ:main}) $\frac{1}{\sqrt{\epsilon}}\int_{\mathbb{R}^{d}}\eta^{\epsilon}(t,x)\mu_{t}^{\epsilon}(dx)$ may be written as $\frac{1}{\sqrt{\epsilon}}\tilde{\mathbb{E}}\eta^{\epsilon}(t,x^{\epsilon}(t))$.
\end{remark}

\section{Tightness of solutions}\label{Tight}
  \setcounter{equation}{0}
  \renewcommand{\theequation}
{3.\arabic{equation}}
Next we give some bounded estimates on $x^{\epsilon}$ and $\dot{x}^{\epsilon}$. We rewrite (\ref{equ:main}) as
\begin{eqnarray}
\dot{x}^{\epsilon}(t)&=&y^{\epsilon}(t)\label{x-equ}\\
\epsilon\dot{y}^{\epsilon}(t)&=&-\alpha y^{\epsilon}(t)-\nabla V(x^{\epsilon}(t),\mu_{t}^{\epsilon})+\frac{1}{\sqrt{\epsilon}}\int_{\mathbb{R}^{d}}\eta^{\epsilon}(t,x)\mu_{t}^{\epsilon}(dx),\label{y-equ}
\end{eqnarray}
then by (\ref{y-equ})
\begin{eqnarray}\label{y-slu}
y^{\epsilon}(t)&=&y_{0}e^{-\frac{\alpha}{\epsilon}t}-\frac{1}{\epsilon}\int_{0}^{t}e^{-\frac{\alpha}{\epsilon}(t-s)}\nabla V(x^{\epsilon}(s),\mu_{s}^{\epsilon})ds\nonumber\\
&&+\frac{1}{\epsilon\sqrt{\epsilon}}\int_{0}^{t}e^{-\frac{\alpha}{\epsilon}(t-s)}\int_{\mathbb{R}^{d}}\eta^{\epsilon}(s,x)\mu_{s}^{\epsilon}(dx)ds,
\end{eqnarray}
and in combination with (\ref{x-equ}) and integration by parts,
\begin{eqnarray}\label{equ:1.10}
x^{\epsilon}(t)&=&x_{0}+\frac{\epsilon}{\alpha}y_{0}(1-e^{-\frac{\alpha}{\epsilon}t})-\frac{1}{\alpha}\int_{0}^{t}(1-e^{-\frac{\alpha}{\epsilon}(t-s)})\nabla V(x^{\epsilon}(s),\mu_{s}^{\epsilon})ds\nonumber\\
&&+\frac{1}{\alpha\sqrt{\epsilon}}\int_{0}^{t}(1-e^{-\frac{\alpha}{\epsilon}(t-s)})\int_{\mathbb{R}^{d}}\eta^{\epsilon}(s,x)\mu_{s}^{\epsilon}(dx)ds\nonumber\\
&\triangleq& I_{1}^{\epsilon}(t)+I_{2}^{\epsilon}(t).
\end{eqnarray}

Then we have the following result.
\begin{lemma}\label{MoEs}
Suppose that $(\mathbf{H_{1}})$-$(\mathbf{H_{3}})$ hold, then there exists $C_{T}>0$ such that 
\begin{eqnarray*}
\sup_{0\leq t\leq T}\mathbb{E}\|\sqrt{\epsilon}x^{\epsilon}(t)\|^{2}&\leq& C_{T},\quad\sup_{0\leq t\leq T}\mathbb{E}\|\sqrt{\epsilon}x^{\epsilon}(t)\|^{4}\leq C_{T}.\label{equ:b3.5}\\
\sup_{0\leq t\leq T}\mathbb{E}\|\sqrt{\epsilon}\dot{x}^{\epsilon}(t)\|^{2}&\leq& C_{T},\quad \sup_{0\leq t\leq T}\mathbb{E}\|\sqrt{\epsilon}\dot{x}^{\epsilon}(t)\|^{4}\leq C_{T}.\label{equ:b3.6}
\end{eqnarray*}
\end{lemma}
\begin{proof}
By (\ref{equ:1.10}), $(\mathbf{H_{1}})$, $(\mathbf{H_{3}})$ and  the H${\rm\ddot{o}}$lder inequality, we have
\begin{eqnarray*}
\mathbb{E}\|\sqrt{\epsilon}x^{\epsilon}(t)\|^{2}&\leq&
8\|x_{0}\|^{2}+8\Big(\frac{\epsilon}{\alpha}\Big)^{2}\|y_{0}\|^{2}+\frac{16TL_{V}^{2}}{\alpha^{2}}\int_{0}^{t}(\mathbb{E}\|\sqrt{\epsilon}x^{\epsilon}(u)\|^{2}+\epsilon\mathcal{W}_{2}(\mu_{u}^{\epsilon},\delta_{0})^{2})du\nonumber\\
&&+\frac{16T^{2}\|\nabla V(0,\delta_{0})\|^{2}}{\alpha^{2}}+\frac{8T}{\alpha^{2}}\int_{0}^{t}\mathbb{E}\|\tilde{\mathbb{E}}\eta^{\epsilon}(s,x^{\epsilon}(s))\|^{2}ds\nonumber\\
&\leq&8\|x_{0}\|^{2}+8\Big(\frac{\epsilon}{\alpha}\Big)^{2}\|y_{0}\|^{2}+\frac{32TL_{V}}{\alpha}^{2}\int_{0}^{t}\mathbb{E}\|\sqrt{\epsilon}x^{\epsilon}(s)\|^{2}ds\nonumber\\
&&+\frac{16T^{2}\|\nabla V(0,\delta_{0})\|^{2}}{\alpha^{2}}+\frac{8T}{\alpha^{2}}\int_{0}^{t}\mathbb{E}\|\tilde{\mathbb{E}}\eta^{\epsilon}(s,x^{\epsilon}(s))\|^{2}ds,
\end{eqnarray*}
by Gronwall's inequality,
\begin{eqnarray}\label{Squ}
\sup_{0\leq t\leq T}\mathbb{E}\|\sqrt{\epsilon}x^{\epsilon}(t)\|^{2}&\leq& C_{T}.
\end{eqnarray}
Furthermore
\begin{eqnarray*}
\mathbb{E}\|\sqrt{\epsilon}x^{\epsilon}(t)\|^{4}&\leq& 64\|x_{0}\|^{4}+64\Big(\frac{\epsilon}{\alpha}\Big)^{4}\|y_{0}\|^{4}+\frac{64}{\alpha^{4}}\mathbb{E}\Big(\sqrt{\epsilon}\int_{0}^{t}(1-e^{-\frac{\alpha}{\epsilon}(t-u)})\|\nabla V(x^{\epsilon}(u),\mu_{u}^{\epsilon})\|du\Big)^{4}\nonumber\\
&&+64\Big(\frac{1}{\alpha}\Big)^{4}\mathbb{E}\Big(\int_{0}^{t}(1-e^{-\frac{\alpha}{\epsilon}(t-u)})\|\tilde{\mathbb{E}}\eta^{\epsilon}(u,x^{\epsilon}(u))\|du\Big)^{4}\nonumber\\
&\leq&64\|x_{0}\|^{4}+64\Big(\frac{\epsilon}{\alpha}\Big)^{4}\|y_{0}\|^{4}+\frac{64}{\alpha^{4}}\mathbb{E}\Big(L_{V}\int_{0}^{t}\|\sqrt{\epsilon}x^{\epsilon}(u)\|du\nonumber\\
&&+L_{V}\int_{0}^{t}(\mathbb{E}\|\sqrt{\epsilon}x^{\epsilon}(u)\|^{2})^{\frac{1}{2}}du+\sqrt{\epsilon}T\|\nabla V(0,\delta_{0})\|\Big)^{4}+64\Big(\frac{1}{\alpha}\Big)^{4}M_{\eta}^{4}T^{4}\nonumber\\
&\leq&64\|x_{0}\|^{4}+64\Big(\frac{\epsilon}{\alpha}\Big)^{4}\|y_{0}\|^{4}+\frac{12^{3}L_{V}^{4}T^{3}}{\alpha^{4}}\int_{0}^{t}\mathbb{E}\|\sqrt{\epsilon}x^{\epsilon}(u)\|^{4}du\nonumber\\
&&+\frac{12^{3}\epsilon^{2}}{\alpha^{4}}\|\nabla V(0,\delta_{0})\|^{4}T^{4}+64\Big(\frac{1}{\alpha}\Big)^{4}M_{\eta}^{4}T^{4},
\end{eqnarray*}
the Gronwall inequality yields
\begin{eqnarray}\label{equ:3.80}
\sup_{0\leq t\leq T}\mathbb{E}\|\sqrt{\epsilon}x^{\epsilon}(t)\|^{4}\leq C_{T}.
\end{eqnarray}
By chain rules and the Cauchy-Schwartz inequality,
\begin{eqnarray*}
\frac{1}{2}\frac{d}{dt}\|\sqrt{\epsilon}\dot{x}^{\epsilon}(t)\|^{2}&=&-\frac{\alpha}{\epsilon}\langle\sqrt{\epsilon}\dot{x}^{\epsilon}(t),\sqrt{\epsilon}\dot{x}^{\epsilon}(t)\rangle-\frac{1}{\sqrt{\epsilon}}\langle\sqrt{\epsilon}\dot{x}^{\epsilon}(t),\nabla V(x^{\epsilon}(t),\mu_{t}^{\epsilon})\rangle\nonumber\\
&&+\frac{1}{\epsilon}\langle\sqrt{\epsilon}\dot{x}^{\epsilon}(t),\tilde{\mathbb{E}}\eta^{\epsilon}(t,x^{\epsilon}(t))\rangle\\
&\leq& -\frac{\alpha}{2\epsilon}\|\sqrt{\epsilon}\dot{x}^{\epsilon}(t)\|^{2}+\frac{2}{\alpha}L_{V}^{2}\|x^{\epsilon}(t)\|^{2}+\frac{2}{\alpha}L_{V}^{2}\mathbb{E}\|x^{\epsilon}(t)\|^{2}\\
&&+\frac{2}{\alpha}\|\nabla V(0,\delta_{0})\|^{2}+\frac{1}{\alpha\epsilon}\|\tilde{\mathbb{E}}\eta^{\epsilon}(t,x^{\epsilon}(t))\|^{2},
\end{eqnarray*}
then taking expectation on both sides and by ($\mathbf{H}_{3}$), (\ref{Squ}) and the Gronwall inequality,
\begin{eqnarray*}
\sup_{0\leq t\leq T}\mathbb{E}\|\sqrt{\epsilon}\dot{x}^{\epsilon}(t)\|^{2}\leq C_{T}.
\end{eqnarray*}
Similarly,
\begin{eqnarray*}
\frac{1}{2}\frac{d}{dt}\|\sqrt{\epsilon}\dot{x}^{\epsilon}(t)\|^{4}&\leq& \|\sqrt{\epsilon}\dot{x}^{\epsilon}(t)\|^{2}\Big(-\frac{\alpha}{2\epsilon}\|\sqrt{\epsilon}\dot{x}^{\epsilon}(t)\|^{2}+\frac{2}{\alpha}L_{V}^{2}\|\sqrt{\epsilon}x^{\epsilon}(t)\|^{2}\nonumber\\
&&+\frac{2}{\alpha}L_{V}^{2}\mathbb{E}\|\sqrt{\epsilon}x^{\epsilon}(t)\|^{2}
+\frac{2}{\alpha}\|\nabla_{x}V(0,\delta_{0})\|^{2}+\frac{1}{\alpha\epsilon}\|\tilde{\mathbb{E}}\eta^{\epsilon}(t,x^{\epsilon}(t))\|^{2}\Big)\nonumber\\
&\leq&-\frac{\alpha}{4\epsilon}\|\sqrt{\epsilon}\dot{x}^{\epsilon}(t)\|^{4}+\frac{16L_{V}^{4}}{\alpha^{3}}\|\sqrt{\epsilon}x^{\epsilon}(t)\|^{4}+\frac{16L_{V}^{4}}{\alpha^{3}}\mathbb{E}\|\sqrt{\epsilon}x^{\epsilon}(t)\|^{4}\nonumber\\
&&+\frac{8}{\alpha^{3}}\|\nabla V(0,\delta_{0})\|^{4}+\frac{1}{\alpha^{3}\epsilon}\|\tilde{\mathbb{E}}\eta^{\epsilon}(t,x^{\epsilon}(t))\|^{4},
\end{eqnarray*}
Gronwall's inequality and (\ref{equ:3.80}) yields
\begin{eqnarray}\label{equ:1.32}
\sup_{0\leq t\leq T}\mathbb{E}\|\sqrt{\epsilon}\dot{x}^{\epsilon}(t)\|^{4}\leq C_{T}.
\end{eqnarray}
\end{proof}

Next we are devoted to derive the tightness of the solution $x^{\epsilon}(t)$. For this, we firstly give some uniformly bounded estimates to $I_{2}^{\epsilon}(t)$ in (\ref{equ:1.10}) with respect to $\epsilon$. Note that
\begin{eqnarray}\label{lee}
I_{2}^{\epsilon}(t)&=&\frac{1}{\alpha\sqrt{\epsilon}}\int_{0}^{t}\int_{\mathbb{R}^{d}}\eta^{\epsilon}(s,x)\mu_{s}^{\epsilon}(dx)ds\nonumber\\
&&-\frac{1}{\alpha\sqrt{\epsilon}}\int_{0}^{t}\int_{\mathbb{R}^{d}}e^{-\frac{\alpha}{\epsilon}(t-s)}\eta^{\epsilon}(s,x)\mu_{s}^{\epsilon}(dx)ds.
\end{eqnarray}
Define 
\begin{eqnarray}\label{lef}
u^{\epsilon}(t)=\frac{1}{\alpha\sqrt{\epsilon}}\int_{0}^{t}\int_{\mathbb{R}^{d}}\eta^{\epsilon}(s,x)\mu_{s}^{\epsilon}(dx)ds,
\end{eqnarray}
then
\begin{eqnarray}
\dot{u}^{\epsilon}(t)&=&\frac{1}{\alpha\sqrt{\epsilon}}\int_{\mathbb{R}^{d}}\eta^{\epsilon}(s,x)\mu_{t}^{\epsilon}(dx),\nonumber\\
&=&\frac{1}{\alpha\sqrt{\epsilon}}\tilde{\mathbb{E}}\eta^{\epsilon}(t,x^{\epsilon}(t)), \quad u^{\epsilon}(0)=0.
\end{eqnarray}
In order to obtain uniformly bounded estimates in $\epsilon$, we construct a martingale to deal with the singular term $\frac{1}{\alpha\sqrt{\epsilon}}\tilde{\mathbb{E}}\eta^{\epsilon}(t,x^{\epsilon}(t))$. By chain rules and the Cauchy-Schwart inequality, we have
\begin{eqnarray*}
\frac{1}{2}\frac{d}{dt}\mathbb{E}\|u^{\epsilon}(t)\|^{2}
&\leq&\frac{1}{\alpha\sqrt{\epsilon}}\mathbb{E}\Big\|\int_{\mathbb{R}^{d}}\eta^{\epsilon}(t,x)\mu_{t}^{\epsilon}(dx)\Big\|^{2}+\frac{1}{\alpha\sqrt{\epsilon}}\mathbb{E}\|u^{\epsilon}(t)\|^{2},
\end{eqnarray*}
by Gronwall's inequality and ($\mathbf{H}_{3}$),
\begin{eqnarray}\label{equ:1.19}
\mathbb{E}\|u^{\epsilon}(t)\|^{2}\leq \frac{1}{\alpha\sqrt{\epsilon}}\int_{0}^{t}e^{\frac{1}{\alpha\sqrt{\epsilon}}(t-s)}\mathbb{E}\Big\|\int_{\mathbb{R}^{d}}\eta^{\epsilon}(s,x)\mu_{s}^{\epsilon}(dx)\Big\|^{2}ds\leq C_{\epsilon}.
\end{eqnarray}
In order to derive the uniform estimates to $\epsilon$, we need to deal with the singular term $\frac{1}{\sqrt{\epsilon}}\langle\int_{\mathbb{R}^{d}}\eta^{\epsilon}(t,x)\mu_{t}^{\epsilon}(dx),u^{\epsilon}(t)\rangle$.
Let 
\begin{eqnarray*}
Y_{1}^{\epsilon}(t)&=&\frac{1}{\sqrt{\epsilon}}\int_{t}^{\infty}\mathbb{E}\Big[\Big\langle\int_{\mathbb{R}^{d}}\eta^{\epsilon}(s,x)\mu_{t}^{\epsilon}(dx),u^{\epsilon}(t)\Big\rangle\Big|\mathcal{F}_{0}^{\frac{t}{\epsilon}}\Big]ds\nonumber\\
&=&\Big\langle\int_{\mathbb{R}^{d}}\eta^{\epsilon}(t,x)\mu_{t}^{\epsilon}(dx),u^{\epsilon}(t)\Big\rangle\frac{1}{\sqrt{\epsilon}}\int_{t}^{\infty}m\Big(\frac{s-t}{\epsilon}\Big)ds\nonumber\\
&=&\sqrt{\epsilon}K\langle\tilde{\mathbb{E}}\eta^{\epsilon}(t,x^{\epsilon}(t)),u^{\epsilon}(t)\rangle,
\end{eqnarray*}
and 
\begin{eqnarray*}
&&Z_{1}^{\epsilon}(t)=\lim_{\delta\to0}\frac{\mathbb{E}[Y_{1}^{\epsilon}(t+\delta)-Y_{1}^{\epsilon}(t)|\mathcal{F}_{0}^{\frac{t}{\epsilon}}]}{\delta}\nonumber\\
&=&\sqrt{\epsilon}K\lim_{\delta\to0}\frac{\mathbb{E}[\langle\tilde{\mathbb{E}}\eta^{\epsilon}(t+\delta,x^{\epsilon}(t+\delta)),u^{\epsilon}(t+\delta)\rangle-\langle\tilde{\mathbb{E}}\eta^{\epsilon}(t+\delta,x^{\epsilon}(t)),u^{\epsilon}(t+\delta)\rangle|\mathcal{F}_{0}^{\frac{t}{\epsilon}}]}{\delta}\nonumber\\
&&+\sqrt{\epsilon}K\lim_{\delta\to0}\frac{\mathbb{E}[\langle\tilde{\mathbb{E}}\eta^{\epsilon}(t+\delta,x^{\epsilon}(t)),u^{\epsilon}(t+\delta)\rangle-\langle\tilde{\mathbb{E}}\eta^{\epsilon}(t+\delta,x^{\epsilon}(t)),u^{\epsilon}(t)\rangle|\mathcal{F}_{0}^{\frac{t}{\epsilon}}]}{\delta}\nonumber\\
&&+\sqrt{\epsilon}K\lim_{\delta\to0}\frac{\mathbb{E}[\langle\tilde{\mathbb{E}}\eta^{\epsilon}(t+\delta,x^{\epsilon}(t)),u^{\epsilon}(t)\rangle-\langle\tilde{\mathbb{E}}\eta^{\epsilon}(t,x^{\epsilon}(t)),u^{\epsilon}(t)\rangle|\mathcal{F}_{0}^{\frac{t}{\epsilon}}]}{\delta}\nonumber\\
&=&\sqrt{\epsilon}K\mathbb{E}\Big[\lim_{\delta\to0}\frac{\langle\tilde{\mathbb{E}}\nabla_{x}\eta^{\epsilon}(t+\delta,x^{\epsilon}(t)+\theta(x^{\epsilon}(t+\delta)-x^{\epsilon}(t)))\cdot(x^{\epsilon}(t+\delta)-x^{\epsilon}(t)),u^{\epsilon}(t+\delta)\rangle}{\delta}\Big|\mathcal{F}_{0}^{\frac{t}{\epsilon}}\Big]\nonumber\\
&&+\sqrt{\epsilon}K\mathbb{E}\Big[\lim_{\delta\to0}\frac{\langle\tilde{\mathbb{E}}\eta^{\epsilon}(t+\delta,x^{\epsilon}(t)),u^{\epsilon}(t+\delta)-u^{\epsilon}(t)\rangle}{\delta}\Big|\mathcal{F}_{0}^{\frac{t}{\epsilon}}\Big]\nonumber\\
&&+\sqrt{\epsilon}K\mathbb{E}\Big[\lim_{\delta\to0}\frac{\langle\tilde{\mathbb{E}}\eta^{\epsilon}(t+\delta,x^{\epsilon}(t))-\tilde{\mathbb{E}}\eta^{\epsilon}(t,x^{\epsilon}(t)),u^{\epsilon}(t)\rangle}{\delta}\Big|\mathcal{F}_{0}^{\frac{t}{\epsilon}}\Big]\nonumber\\
&=&\sqrt{\epsilon}K\langle\tilde{\mathbb{E}}\nabla_{x}\eta^{\epsilon}(t,x^{\epsilon}(t))\cdot\dot{x}^{\epsilon}(t),u^{\epsilon}(t)\rangle+\frac{K}{\alpha}\|\tilde{\mathbb{E}}\eta^{\epsilon}(t,x^{\epsilon}(t))\|^{2}\nonumber\\
&&-\frac{K\beta}{\sqrt{\epsilon}}\langle\tilde{\mathbb{E}}\eta^{\epsilon}(t,x^{\epsilon}(t)),u^{\epsilon}(t)\rangle,\quad \beta\triangleq-m'(0)\geq0.
\end{eqnarray*}
Then by Proposition \ref{proa1}, we have
\begin{eqnarray}
&&Y_{1}^{\epsilon}(t)-\int_{0}^{t}Z_{1}^{\epsilon}(s)ds=\sqrt{\epsilon}K\langle\tilde{\mathbb{E}}\eta^{\epsilon}(t,x^{\epsilon}(t)),u^{\epsilon}(t)\rangle\nonumber\\
&&-\sqrt{\epsilon}K\int_{0}^{t}\langle\tilde{\mathbb{E}}\nabla_{x}\eta^{\epsilon}(s,x^{\epsilon}(s))\cdot\dot{x}^{\epsilon}(s),u^{\epsilon}(s)\rangle ds\nonumber\\
&&-\frac{K}{\alpha}\int_{0}^{t}\|\tilde{\mathbb{E}}\eta^{\epsilon}(s,x^{\epsilon}(s))\|^{2}ds+\frac{K\beta}{\sqrt{\epsilon}}\int_{0}^{t}\langle\tilde{\mathbb{E}}\eta^{\epsilon}(s,x^{\epsilon}(s)),u^{\epsilon}(s)\rangle ds
\end{eqnarray}
is a $\mathcal{F}_{0}^{\frac{t}{\epsilon}}$-martingale. Now we define a process
\begin{eqnarray}
M_{1}^{\epsilon}(t)&=&\sqrt{\epsilon}K\langle\tilde{\mathbb{E}}\eta^{\epsilon}(t,x^{\epsilon}(t)),u^{\epsilon}(t)\rangle-\sqrt{\epsilon}K\int_{0}^{t}\langle\tilde{\mathbb{E}}\nabla_{x}\eta^{\epsilon}(s,x^{\epsilon}(s))\cdot\dot{x}^{\epsilon}(s),u^{\epsilon}(s)\rangle ds\nonumber\\
&&-\frac{K}{\alpha}\int_{0}^{t}\|\tilde{\mathbb{E}}\eta^{\epsilon}(s,x^{\epsilon}(s))\|^{2}ds+\frac{K\beta}{\sqrt{\epsilon}}\int_{0}^{t}\langle\tilde{\mathbb{E}}\eta^{\epsilon}(s,x^{\epsilon}(s)),u^{\epsilon}(s)\rangle ds,\nonumber
\end{eqnarray}
then $M_{1}^{\epsilon}(t)$ is a $\mathcal{F}_{0}^{\frac{t}{\epsilon}}$-martingale and $\mathbb{E}M_{1}^{\epsilon}(t)=0$. For any fixed $\epsilon$, by (\ref{equ:1.19}) and $(\mathbf{H_{3}})$, $\{M_{1}^{\epsilon}(t)\}$ is a square integrable martingale. 
Note that
\begin{eqnarray*}
dM_{1}^{\epsilon}(t)&=&d[\sqrt{\epsilon}K\langle\tilde{\mathbb{E}}\eta^{\epsilon}(t,x^{\epsilon}(t)),u^{\epsilon}(t)\rangle]-\frac{K}{\alpha}\|\tilde{\mathbb{E}}\eta^{\epsilon}(t,x^{\epsilon}(t))\|^{2}dt\nonumber\\
&&+\frac{K\beta}{\sqrt{\epsilon}}\langle\tilde{\mathbb{E}}\eta^{\epsilon}(t,x^{\epsilon}(t)),u^{\epsilon}(t)\rangle dt-\sqrt{\epsilon}K\langle\tilde{\mathbb{E}}\nabla_{x}\eta^{\epsilon}(t,x^{\epsilon}(t))\cdot\dot{x}^{\epsilon}(t),u^{\epsilon}(t)\rangle dt,
\end{eqnarray*}
then
\begin{eqnarray}\label{equ:3.150}
&&\frac{K\beta}{2}d\|u^{\epsilon}(t)\|^{2}+d[\sqrt{\epsilon}K\langle\tilde{\mathbb{E}}\eta^{\epsilon}(t,x^{\epsilon}(t)),u^{\epsilon}(t)\rangle]=dM_{1}^{\epsilon}(t)+\frac{K}{\alpha}\|\tilde{\mathbb{E}}\eta^{\epsilon}(t,x^{\epsilon}(t))\|^{2}dt\nonumber\\
&&\quad\quad\quad\quad\quad\quad+\sqrt{\epsilon}K\langle\tilde{\mathbb{E}}\nabla_{x}\eta^{\epsilon}(t,x^{\epsilon}(t))\cdot\dot{x}^{\epsilon}(t),u^{\epsilon}(t)\rangle dt,
\end{eqnarray}
integrating from $0$ to $t$, taking expectation and by Young's inequality, we have 
\begin{eqnarray*}
\frac{K\beta}{4}\mathbb{E}\|u^{\epsilon}(t)\|^{2}&\leq& \frac{\epsilon}{2\beta}\mathbb{E}\|\tilde{\mathbb{E}}\eta^{\epsilon}(t,x^{\epsilon}(t))\|^{2}+\frac{K}{\alpha}\mathbb{E}\|\tilde{\mathbb{E}}\eta^{\epsilon}(t,x^{\epsilon}(t))\|^{2}\nonumber\\
&&+\frac{K}{\beta}\mathbb{E}\|\sqrt{\epsilon}\dot{x}^{\epsilon}(t)\|^{4}+\frac{K}{\beta}\mathbb{E}\|\tilde{\mathbb{E}}\nabla_{x}\eta^{\epsilon}(t,x^{\epsilon}(t))\|^{4},
\end{eqnarray*}
this together with ($\mathbf{H}_{3}$) and Lemma \ref{MoEs} yield
\begin{eqnarray}\label{equ:1.37}
\sup_{0\leq t\leq T}\mathbb{E}\|u^{\epsilon}(t)\|^{2}\leq C_{T}.
\end{eqnarray} 
Furthermore, we need some estimates on $\|u^{\epsilon}(t)\|^{4}$. By chain rules and (\ref{equ:3.150}),
\begin{eqnarray*}
&&d\Big[\frac{K\beta}{4}\|u^{\epsilon}(t)\|^{4}+\sqrt{\epsilon}K\|u^{\epsilon}(t)\|^{2}\langle u^{\epsilon}(t),\tilde{\mathbb{E}}\eta^{\epsilon}(t,x^{\epsilon}(t))\rangle\Big]\nonumber\\
&&=\|u^{\epsilon}(t)\|^{2}d\Big[\frac{K\beta}{2}\|u^{\epsilon}(t)\|^{2}+\sqrt{\epsilon}K\langle u^{\epsilon}(t),\tilde{\mathbb{E}}\eta^{\epsilon}(t,x^{\epsilon}(t))\rangle\Big]\nonumber\\
&&\quad+\sqrt{\epsilon}K\langle u^{\epsilon}(t),\tilde{\mathbb{E}}\eta^{\epsilon}(t,x^{\epsilon}(t))\rangle d\|u^{\epsilon}(t)\|^{2}\nonumber\\
&&=\|u^{\epsilon}(t)\|^{2}dM_{1}^{\epsilon}(t)+\frac{K}{\alpha}\|u^{\epsilon}(t)\|^{2}\|\tilde{\mathbb{E}}\eta^{\epsilon}(t,x^{\epsilon}(t))\|^{2}dt\nonumber\\
&&\quad+\sqrt{\epsilon}K\|u^{\epsilon}(t)\|^{2}\langle\tilde{\mathbb{E}}\nabla_{x}\eta^{\epsilon}(t,x^{\epsilon}(t))\cdot\dot{x}^{\epsilon}(t),u^{\epsilon}(t)\rangle dt\nonumber\\
&&\quad+2K|\langle u^{\epsilon}(t),\tilde{\mathbb{E}}\eta^{\epsilon}(t,x^{\epsilon}(t))\rangle|^{2}\nonumber\\
&&\leq \|u^{\epsilon}(t)\|^{2}dM_{1}^{\epsilon}(t)+\frac{3K}{\alpha}\|u^{\epsilon}(t)\|^{2}\|\tilde{\mathbb{E}}\eta^{\epsilon}(t,x^{\epsilon}(t))\|^{2}dt\nonumber\\
&&\quad+K\|u^{\epsilon}(t)\|^{3}\|\tilde{\mathbb{E}}\nabla_{x}\eta^{\epsilon}(t,x^{\epsilon}(t))\cdot\sqrt{\epsilon}\dot{x}^{\epsilon}(t)\|dt.
\end{eqnarray*}
And by (\ref{equ:1.37}) and $M_{1}^{\epsilon}(t)$ is a square integrable martingale, we derive that $\int_{0}^{t}\|u^{\epsilon}(s)\|^{2}dM_{1}^{\epsilon}(s)$ is also a square integrable martingale, then
\begin{eqnarray*}
&&d\Big[\frac{K\beta}{4}\mathbb{E}\|u^{\epsilon}(t)\|^{4}+\sqrt{\epsilon}K\mathbb{E}\|u^{\epsilon}(t)\|^{2}\langle u^{\epsilon}(t),\tilde{\mathbb{E}}\eta^{\epsilon}(t,x^{\epsilon}(t))\rangle\Big]\nonumber\\
&&\leq \frac{3K}{2\alpha}\mathbb{E}\|u^{\epsilon}(t)\|^{4}dt+\frac{3K}{2\alpha}\mathbb{E}\|\tilde{\mathbb{E}}\eta^{\epsilon}(t,x^{\epsilon}(t))\|^{4}dt+\frac{3\alpha}{4}\mathbb{E}\|u^{\epsilon}(t)\|^{4}\nonumber\\
&&\quad+\frac{K}{4}\mathbb{E}\|\tilde{\mathbb{E}}\nabla_{x}\eta^{\epsilon}(t,x^{\epsilon}(t))\cdot\sqrt{\epsilon}\dot{x}^{\epsilon}(t)\|^{4}dt\nonumber\\
&&\leq\Big(\frac{3K}{2\alpha}+\frac{3\alpha}{4}\Big)\mathbb{E}\|u^{\epsilon}(t)\|^{4}dt+\frac{3K}{2\alpha}\mathbb{E}\|\tilde{\mathbb{E}}\eta^{\epsilon}(t,x^{\epsilon}(t))\|^{4}dt+\frac{M_{\eta}K}{4}\mathbb{E}\|\sqrt{\epsilon}\dot{x}^{\epsilon}(t)\|^{4},
\end{eqnarray*}
integrating from $0$ to $t$ and taking expectation, together with (\ref{equ:1.32}) and ($\mathbf{H}_{3}$), we have
\begin{eqnarray*}
\frac{K\beta}{8}\mathbb{E}\|u^{\epsilon}(t)\|^{4}&\leq& \Big(\frac{3K}{2\alpha}+\frac{3\alpha}{4}\Big)\int_{0}^{t}\mathbb{E}\|u^{\epsilon}(s)\|^{4}ds+\frac{3K}{2\alpha}\int_{0}^{t}\mathbb{E}\|\tilde{\mathbb{E}}\eta^{\epsilon}(s,x^{\epsilon}(s))\|^{4}ds\nonumber\\
&&+\frac{CK}{4}\int_{0}^{t}\mathbb{E}\|\sqrt{\epsilon}\dot{x}^{\epsilon}(s)\|^{4}ds,
\end{eqnarray*}
Gronwall's inequality yields
\begin{eqnarray}\label{equ:b4}
\sup_{0\leq t\leq T}\mathbb{E}\|u^{\epsilon}(t)\|^{4}\leq C_{T}.
\end{eqnarray}
Now let 
\begin{eqnarray}
\rho^{\epsilon}(t)=\frac{K\beta}{2}\|u^{\epsilon}(t)\|^{2}+\sqrt{\epsilon}K\langle u^{\epsilon}(t),\tilde{\mathbb{E}}\eta^{\epsilon}(t,x^{\epsilon}(t))\rangle,\nonumber
\end{eqnarray}
then by It${\rm \hat{o}}$'s formula and (\ref{equ:3.150}), we have
\begin{eqnarray}
d(\rho^{\epsilon}(t))^{2}&=&2\rho^{\epsilon}(t)d\rho^{\epsilon}(t)+d\langle M_{1}^{\epsilon},M_{1}^{\epsilon}\rangle(t)\nonumber\\
&=&2\rho^{\epsilon}(t)dM_{1}^{\epsilon}(t)+2\rho^{\epsilon}(t)\cdot\frac{K}{\alpha}\|\tilde{\mathbb{E}}\eta^{\epsilon}(t,x^{\epsilon}(t))\|^{2}\nonumber\\
&+&2K\rho^{\epsilon}(t)\langle u^{\epsilon}(t),\tilde{\mathbb{E}}\nabla_{x}\eta^{\epsilon}(t,x^{\epsilon}(t))\cdot\sqrt{\epsilon}\dot{x}^{\epsilon}(t)\rangle dt+d\langle M_{1}^{\epsilon},M_{1}^{\epsilon}\rangle(t)\nonumber,
\end{eqnarray}
where $\langle M_{1}^{\epsilon},M_{1}^{\epsilon}\rangle(t)\triangleq\langle M_{1}^{\epsilon}\rangle(t)$ is the square variation of $M_{1}^{\epsilon}(t)$.
Then
\begin{eqnarray}\label{equ:3.180}
\langle M_{1}^{\epsilon}\rangle(t)&=&(\rho^{\epsilon}(t))^{2}-2\int_{0}^{t}\rho^{\epsilon}(s)dM_{1}^{\epsilon}(s)-\frac{2K}{\alpha}\int_{0}^{t}\rho^{\epsilon}(s)\|\tilde{\mathbb{E}}\eta^{\epsilon}(s,x^{\epsilon}(s))\|^{2}ds\nonumber\\
&&-2K\int_{0}^{t}\rho^{\epsilon}(s)\langle u^{\epsilon}(s),\tilde{\mathbb{E}}\nabla_{x}\eta^{\epsilon}(s,x^{\epsilon}(s))\cdot\sqrt{\epsilon}\dot{x}^{\epsilon}(s)\rangle ds,
\end{eqnarray}
and by (\ref{equ:b4}) as well as $(\mathbf{H_{3}})$
\begin{eqnarray}\label{equ:3.190}
\mathbb{E}(\rho^{\epsilon}(t))^{2}&=&\mathbb{E}\Big(\frac{-K\beta}{2}\|u^{\epsilon}(t)\|^{2}+\sqrt{\epsilon}K\langle u^{\epsilon}(t),\tilde{\mathbb{E}}\eta^{\epsilon}(t,x^{\epsilon}(t))\rangle\Big)^{2}\nonumber\\
&\leq& C_{T},
\end{eqnarray}
Note that
\begin{eqnarray}
\mathbb{E}\int_{0}^{t}\rho^{\epsilon}(s)\|\tilde{\mathbb{E}}\eta^{\epsilon}(s,x^{\epsilon}(s))\|^{2}ds&\leq& \int_{0}^{t}(\mathbb{E}\rho^{\epsilon}(s)^{2})^{\frac{1}{2}}(\mathbb{E}(\|\tilde{\mathbb{E}}\eta^{\epsilon}(s,x^{\epsilon}(s))\|^{4}))^{\frac{1}{2}}ds\nonumber\\
&\leq& C_{T},
\end{eqnarray}
and by (\ref{equ:b3.6}), (\ref{equ:b4}), (\ref{equ:3.190}), we have
\begin{eqnarray}
&&\mathbb{E}\int_{0}^{t}\rho^{\epsilon}(s)\langle u^{\epsilon}(s),\tilde{\mathbb{E}}\nabla_{x}\eta^{\epsilon}(s,x^{\epsilon}(s))\cdot\sqrt{\epsilon}\dot{x}^{\epsilon}(s)\rangle ds\nonumber\\
&&\leq\frac{M_{\eta}}{2}\mathbb{E}\int_{0}^{t}\rho^{\epsilon}(s)^{2}ds+\frac{M_{\eta}}{4}\mathbb{E}\int_{0}^{t}\|\sqrt{\epsilon}\dot{x}^{\epsilon}(s)\|^{4}ds+\frac{M_{\eta}}{4}\mathbb{E}\int_{0}^{t}\|u^{\epsilon}(s)\|^{4}ds\nonumber\\
&&\leq C_{T}.
\end{eqnarray}
Then by (\ref{equ:3.180})
\begin{eqnarray}
\sup_{0\leq t\leq T}\mathbb{E}\langle M_{1}^{\epsilon}\rangle(t)\leq C_{T}.
\end{eqnarray}
By the Burkh${\rm \ddot{o}}$lder-Davies-Gundy inequality, we have
\begin{eqnarray}\label{EBC}
\mathbb{E}\sup_{0\leq t\leq T}|M_{1}^{\epsilon}(t)|^{2}\leq C_{2}\mathbb{E}\langle M_{1}^{\epsilon}\rangle(t)\leq C_{T}.
\end{eqnarray}
Furthermore, by (\ref{equ:3.150}) and the H${\rm\ddot{o}}$lder inequality, we have
\begin{eqnarray*}
&&d\|u^{\epsilon}(t)\|^{2}\\
&\leq& \frac{4}{K\beta}dM_{1}^{\epsilon}(t)+\frac{4}{\alpha\beta}\|\tilde{\mathbb{E}}\eta^{\epsilon}(t,x^{\epsilon}(t))\|^{2}dt+\frac{4}{\beta}\|\tilde{\mathbb{E}}\nabla_{x}\eta^{\epsilon}(t,x^{\epsilon}(t))\|^{2}\|\sqrt{\epsilon}\dot{x}^{\epsilon}(t)\|\|u^{\epsilon}(t)\|dt\nonumber\\
&\leq&\frac{4}{K\beta}dM_{1}^{\epsilon}(t)+\frac{4}{\alpha\beta}\|\tilde{\mathbb{E}}\eta^{\epsilon}(t,x^{\epsilon}(t))\|^{2}dt+\frac{2M_{\eta}}{\beta}\|\sqrt{\epsilon}\dot{x}^{\epsilon}(t)\|^{2}dt+\frac{2M_{\eta}}{\beta}\|u^{\epsilon}(t)\|^{2}dt.
\end{eqnarray*}
Taking the upper bound and then taking expectation on both sides, by (\ref{MoEs}), $(\mathbf{H_{3}})$, (\ref{EBC}) and Gronwall's inequality,
\begin{eqnarray}\label{equ:ub3.24}
\mathbb{E}\sup_{0\leq t\leq T}\|u^{\epsilon}(t)\|^{2}&\leq&C_{T}(\mathbb{E}\sup_{0\leq t\leq T}|M_{1}^{\epsilon}(t)|^{2})^{\frac{1}{2}}+C_{T}\sup_{0\leq t\leq T}\mathbb{E}\|\tilde{\mathbb{E}}\eta^{\epsilon}(t,x^{\epsilon}(t))\|^{2}\nonumber\\
&&+C_{T}\sup_{0\leq t\leq T}\mathbb{E}\|\sqrt{\epsilon}\dot{x}^{\epsilon}(t)\|^{2}\nonumber\\ 
&\leq& C_{T}.
\end{eqnarray}
Further,
\begin{eqnarray*}
u^{\epsilon}(t)-u^{\epsilon}(s)=\frac{1}{\sqrt{\epsilon}}\int_{s}^{t}\tilde{\mathbb{E}}\eta^{\epsilon}(u,x^{\epsilon}(u))du,
\end{eqnarray*}
and by the similar proof as lemma $2.1$ in Watanabe~\cite{Wa},
\begin{eqnarray}\label{equ:uc3.25}
&&\mathbb{E}\|u^{\epsilon}(t)-u^{\epsilon}(s)\|^{4}=\mathbb{E}\Big\|\frac{1}{\sqrt{\epsilon}}\int_{s}^{t}\tilde{\mathbb{E}}\eta^{\epsilon}(u,x^{\epsilon}(u))du\Big\|^{4}\nonumber\\
&=&\frac{3!}{\epsilon^{2}}\int_{s}^{t}du_{1}\int_{s}^{u_{1}}du_{2}\int_{s}^{u_{2}}du_{3}\int_{s}^{u_{3}}\mathbb{E}\langle\tilde{\mathbb{E}}\eta^{\epsilon}(u_{1},x^{\epsilon}(u_{1})),\tilde{\mathbb{E}}\eta^{\epsilon}(u_{2},x^{\epsilon}(u_{2}))\rangle\nonumber\\
&&\cdot\langle\tilde{\mathbb{E}}\eta^{\epsilon}(u_{3},x^{\epsilon}(u_{3})),\tilde{\mathbb{E}}\eta^{\epsilon}(u_{4},x^{\epsilon}(u_{4}))\rangle du_{4}\nonumber\\
&\leq&\frac{C}{\epsilon}\int_{s}^{t}du_{1}\int_{s}^{u_{1}}du_{2}\int_{s}^{u_{2}}du_{3}\int_{s}^{u_{3}}m\Big(\frac{u_{3}-u_{4}}{\epsilon}\Big)^{\frac{1}{2}}m\Big(\frac{u_{1}-u_{2}}{\epsilon}\Big)^{\frac{1}{2}}du_{4}\nonumber\\
&\leq&C_{T}|t-s|\Big[\int_{0}^{\infty}m(s)^{\frac{1}{2}}ds\Big]^{2}\leq C_{T}|t-s|.
\end{eqnarray}
Then by (\ref{equ:ub3.24}) and (\ref{equ:uc3.25}), we have the following result,
\begin{lemma}\label{Lemu}
$\{u^{\epsilon}(t)\}_{0\leq\epsilon\leq1}$ is tight in space $C(0,T;\mathbb{R}^{d})$.
\end{lemma}

Let 
\begin{eqnarray}\label{leg}
v^{\epsilon}(t)=\frac{1}{\alpha\sqrt{\epsilon}}\int_{0}^{t}e^{-\frac{\alpha}{\epsilon}(t-s)}\tilde{\mathbb{E}}\eta^{\epsilon}(s,x^{\epsilon}(s))ds,
\end{eqnarray}
then
\begin{eqnarray}\label{equ:1.54}
&&\mathbb{E}\|v^{\epsilon}(t)\|^{2}=\mathbb{E}\Big\|\frac{1}{\alpha\sqrt{\epsilon}}\int_{0}^{t}e^{-\frac{\alpha}{\epsilon}(t-s)}\tilde{\mathbb{E}}\eta^{\epsilon}(s,x^{\epsilon}(s))ds\Big\|^{2}\nonumber\\
&&=\frac{1}{\alpha^{2}\epsilon}\int_{0}^{t}\int_{0}^{t}e^{-\frac{\alpha}{\epsilon}(t-s)}e^{-\frac{\alpha}{\epsilon}(t-r)}\mathbb{E}\langle\tilde{\mathbb{E}}\eta^{\epsilon}(s,x^{\epsilon}(s)),\tilde{\mathbb{E}}\eta^{\epsilon}(r,x^{\epsilon}(r))\rangle dsdr\nonumber\\
&&\leq\frac{C}{\alpha^{2}\epsilon}\int_{0}^{t}dr\int_{r}^{t}e^{-\frac{\alpha}{\epsilon}(t-s)}e^{-\frac{\alpha}{\epsilon}(t-r)}m\Big(\frac{s-r}{\epsilon}\Big)ds\nonumber\\
&&\leq\frac{CK}{\alpha^{2}}\int_{0}^{t}e^{-\frac{\alpha}{\epsilon}(t-r)}dr\leq \epsilon\frac{CK}{\alpha^{2}}\to0,\quad\epsilon\to0.
\end{eqnarray}
Then $\{v^{\epsilon}(t)\}$ is tight in space $C(0,T;\mathbb{R}^{d})$. Note that $I_{2}^{\epsilon}(t)=u^{\epsilon}(t)-v^{\epsilon}(t)$, then $\{I_{2}^{\epsilon}(t)\}$ is tight in space $C(0,T;\mathbb{R}^{d})$. By $(\mathbf{H_{1}})$-$(\mathbf{H_{3}})$,
\begin{eqnarray}\label{equ:1.11}
&&\|I_{1}^{\epsilon}(t)\|^{2}\nonumber\\
&\leq&4\|x_{0}\|^{2}+4\Big(\frac{\epsilon}{\alpha}\Big)^{2}\|y_{0}\|^{2}+\frac{4}{\alpha^{2}}\Big(\int_{0}^{t}(1-e^{-\frac{\alpha}{\epsilon}(t-u)})\|\nabla V(x^{\epsilon}(u),\mu_{u}^{\epsilon})\|du\Big)^{2}\nonumber\\
&\leq&4\|x_{0}\|^{2}+4\Big(\frac{\epsilon}{\alpha}\Big)^{2}\|y_{0}\|^{2}+\frac{8L_{V}^{2}T}{\alpha^{2}}\int_{0}^{t}\|x^{\epsilon}(u)\|^{2}du\nonumber\\
&&+\frac{8L_{V}^{2}T}{\alpha^{2}}\int_{0}^{t}\mathbb{E}\|x^{\epsilon}(u)\|^{2}du+\frac{8T^{2}}{\alpha^{2}}\|\nabla V(0,\delta_{0})\|^{2}.
\end{eqnarray}
By $(\ref{equ:1.10})$, Gronwall's inequality and the tightness of $I_{2}^{\epsilon}(t)$, one may easily obtain
\begin{eqnarray}\label{equ:x2}
\sup_{0\leq t\leq T}\mathbb{E}\|x^{\epsilon}(t)\|^{2}\leq C_{T}.
\end{eqnarray}
Combining (\ref{equ:1.11}) with (\ref{equ:x2}), we have
\begin{eqnarray}\label{equ:1.55}
\mathbb{E}\sup_{0\leq t\leq T}\|I_{1}^{\epsilon}(t)\|^{2}\leq C_{T}.
\end{eqnarray}
Note that
\begin{eqnarray}\label{equ:1.56}
I_{1}^{\epsilon}(t)&=&x_{0}+\frac{\epsilon}{\alpha}y_{0}(1-e^{-\frac{\alpha}{\epsilon}t})-\frac{1}{\alpha}\int_{0}^{t}(1-e^{-\frac{\alpha}{\epsilon}(t-s)})\nabla V(x^{\epsilon}(s),\mu_{s}^{\epsilon})ds.\nonumber\\
&\triangleq& x_{0}+I_{1,1}^{\epsilon}(t)+I_{1,2}^{\epsilon}(t)
\end{eqnarray}
By a basic inequality, $|e^{x}-e^{y}|\leq|x-y|(e^{x}+e^{y}),~ x,y\in \mathbb{R}$, we have
\begin{eqnarray}\label{equ:1.57}
\mathbb{E}\|I_{1,1}^{\epsilon}(t)-I_{1,1}^{\epsilon}(s)\|^{4}\leq\Big(\frac{\epsilon}{\alpha}\Big)^{4}\|y_{0}\|^{4}|e^{-\frac{\alpha}{\epsilon}t}-e^{-\frac{\alpha}{\epsilon}s}|^{4}\leq 16\|y_{0}\|^{4}|t-s|^{4}.
\end{eqnarray}
Direct calculation yields,
\begin{eqnarray}\label{equ:1.58}
\mathbb{E}\|I_{1,2}^{\epsilon}(t)-I_{1,2}^{\epsilon}(s)\|^{4}&\leq&\frac{8}{\alpha^{4}}\mathbb{E}\Big\|\int_{s}^{t}(1-e^{-\frac{\alpha}{\epsilon}(t-u)})\nabla V(x^{\epsilon}(u),\mu_{u}^{\epsilon})du\Big\|^{4}\nonumber\\
&&+\frac{8}{\alpha^{4}}\mathbb{E}\Big\|\int_{0}^{s}(e^{-\frac{\alpha}{\epsilon}(t-u)}-e^{-\frac{\alpha}{\epsilon}(s-u)})\nabla V(x^{\epsilon}(u),\mu_{u}^{\epsilon})du\Big\|^{4}\nonumber\\
&\triangleq&I_{1,2,1}^{\epsilon}(t)+I_{1,2,2}^{\epsilon}(t).
\end{eqnarray}
By the H${\rm\ddot{o}}$lder inequality and $(\mathbf{H_{1}})$,
\begin{eqnarray}\label{Iequ1}
I_{1,2,1}^{\epsilon}(t)&\leq&\frac{8}{\alpha^{4}}\Big(\int_{s}^{t}(1-e^{-\frac{\alpha}{\epsilon}(t-u)})^{\frac{4}{3}}du\Big)^{3}\Big(\int_{s}^{t}\mathbb{E}\|\nabla V(x^{\epsilon}(u),\mu_{u}^{\epsilon})\|^{4}du\Big)\nonumber\\
&\leq&\frac{8^{2}}{\alpha^{4}}|t-s|^{3}\Big[16L_{V}^{4}\int_{s}^{t}\mathbb{E}\|x^{\epsilon}(u)\|^{4}du+\|\nabla V(0,\delta_{0}\|^{4}T\Big].
\end{eqnarray}
Similarly,
\begin{eqnarray}\label{Iequ2}
&&I_{1,2,2}^{\epsilon}(t)\leq\frac{8}{\alpha^{4}}\mathbb{E}\Big[\int_{0}^{s}|e^{-\frac{\alpha}{\epsilon}(t-u)}-e^{-\frac{\alpha}{\epsilon}(s-u)}|\|\nabla V(x^{\epsilon}(u),\mu_{u}^{\epsilon})\|du\Big]^{4}\\
&\leq&\frac{8}{\alpha^{4}}\Big(\int_{0}^{s}|e^{-\frac{\alpha}{\epsilon}(t-u)}-e^{-\frac{\alpha}{\epsilon}(s-u)}|du\Big)\Big(\int_{0}^{s}|e^{-\frac{\alpha}{\epsilon}(t-u)}-e^{-\frac{\alpha}{\epsilon}(s-u)}|\|\nabla V(x^{\epsilon}(u),\mu_{u}^{\epsilon})\|^{\frac{4}{3}}du\Big)^{3}.\nonumber
\end{eqnarray}
By a basic inequality, $|e^{x}-e^{y}|\leq|x-y|(e^{x}+e^{y}),~ x,y\in \mathbb{R}$,
\begin{eqnarray}\label{Iequ3}
\int_{0}^{s}|e^{-\frac{\alpha}{\epsilon}(t-u)}-e^{-\frac{\alpha}{\epsilon}(s-u)}|du
&\leq&\frac{2\alpha}{\epsilon}|t-s|\int_{0}^{s}e^{-\frac{\alpha}{\epsilon}(s-u)}du\leq 2|t-s|.\nonumber\\
\end{eqnarray}
By $(\mathbf{H_{1}})$,
\begin{eqnarray}\label{Iequ4}
&&\Big(\int_{0}^{s}|e^{-\frac{\alpha}{\epsilon}(t-u)}-e^{-\frac{\alpha}{\epsilon}(s-u)}|\|\nabla V(x^{\epsilon}(u),\mu_{u}^{\epsilon})\|^{\frac{4}{3}}du\Big)^{3}\nonumber\\
&\leq&CL_{V}^{4}T^{2}\int_{0}^{s}\mathbb{E}\|x^{\epsilon}(u)\|^{4}du+CL_{V}^{4}T^{3}\|\nabla V(0,\delta_{0})\|^{4}.
\end{eqnarray}
Then by (\ref{Iequ2})-(\ref{Iequ4}), we obtain,
\begin{eqnarray}\label{Iequ5}
I_{1,2,2}^{\epsilon}(t)&\leq& \frac{C}{\alpha^{4}}|t-s|\Big(L_{V}^{4}T^{2}\int_{0}^{s}\mathbb{E}\|x^{\epsilon}(u)\|^{4}du+L_{V}^{4}T^{3}\|\nabla V(0,\delta_{0})\|^{4}\Big).\nonumber\\
\end{eqnarray}
Combining (\ref{equ:1.58}) with (\ref{Iequ1}) and (\ref{Iequ5}),
\begin{eqnarray}\label{Iequ6}
\mathbb{E}\|I_{1,2}^{\epsilon}(t)-I_{1,2}^{\epsilon}(s)\|^{4}&\leq&C_{T}\Big(1+\int_{s}^{t}\mathbb{E}\|x^{\epsilon}(u)\|^{4}du\Big)|t-s|.
\end{eqnarray}
Next we estimate $\mathbb{E}\|x^{\epsilon}(t)\|^{4}$. In fact, by (\ref{equ:1.10}), (\ref{lee}), (\ref{lef}), (\ref{leg}) and (\ref{equ:b4}), we have
\begin{eqnarray}\label{Iequ7}
\mathbb{E}\|x^{\epsilon}(t)\|^{4}
&\leq& C\|x_{0}\|^{4}+C\Big(\frac{\epsilon}{\alpha}\Big)^{4}\|y_{0}\|^{4}+C\Big(\frac{1}{\alpha}\Big)^{4}T^{3}L_{V}^{4}\int_{0}^{t}\mathbb{E}\|x^{\epsilon}(u)\|^{4}du\nonumber\\
&&+C\Big(\frac{1}{\alpha}\Big)^{4}T^{4}\|\nabla V(0,\delta_{0})\|^{4}+C_{T}+C\mathbb{E}\|v^{\epsilon}(t)\|^{4}.
\end{eqnarray}
Note that
\begin{eqnarray}\label{lequ8}
\mathbb{E}\|v^{\epsilon}(t)\|^{4}\leq\mathbb{E}\Big[\frac{1}{\alpha\sqrt{\epsilon}}\int_{0}^{t}e^{-\frac{\alpha}{\epsilon}(t-u)}\|\tilde{\mathbb{E}}\eta^{\epsilon}(u,x^{\epsilon}(u))\|du\Big]^{4}
\leq \frac{M_{\eta}^{4}}{\alpha^{8}}\epsilon^{2}.
\end{eqnarray}
Then by (\ref{Iequ7}), (\ref{lequ8}) and Gronwall's inequality, we have
\begin{eqnarray}\label{Iequ9}
\sup_{0\leq t\leq T}\mathbb{E}\|x^{\epsilon}(t)\|^{4}\leq C_{T},
\end{eqnarray}
and together with (\ref{Iequ6}), we have
\begin{eqnarray}\label{Iequ10}
\mathbb{E}\|I_{1,2}^{\epsilon}(t)-I_{1,2}^{\epsilon}(s)\|^{4}&\leq&C_{T}|t-s|.
\end{eqnarray}
Thus
\begin{eqnarray}\label{Iequ11}
\mathbb{E}\|I_{1}^{\epsilon}(t)-I_{1}^{\epsilon}(s)\|^{4}&\leq&C\mathbb{E}\|I_{1,1}^{\epsilon}(t)-I_{1,1}^{\epsilon}(s)\|^{4}+C\mathbb{E}\|I_{1,2}^{\epsilon}(t)-I_{1,2}^{\epsilon}(s)\|^{4}\nonumber\\
&\leq& C_{T}|t-s|,
\end{eqnarray}
together with (\ref{equ:1.55}), we obtain $\{I_{1}^{\epsilon}(t)\}$ is tight in $C(0,T;\mathbb{R}^{d})$. By (\ref{equ:1.10}), (\ref{Lemu}) and (\ref{lequ8}), we obtain that
\begin{theorem}
$\{x^{\epsilon}(t)\}_{0<\epsilon\leq1}$ is tight in space  $C(0,T;\mathbb{R}^{d})$.
\end{theorem}

\section{Diffusion approximation}\label{DA}
  \setcounter{equation}{0}
  \renewcommand{\theequation}
{4.\arabic{equation}}
We apply a diffusion approximation to derive the limit of the solution $x^{\epsilon}(t)$ to equation (\ref{equ:main}). By the tightness of $\{x^{\epsilon}\}$, there is a subsequence, which we still write as $\{x^{\epsilon}\}$, converges in distribution to $x\in C(0,T;\mathbb{R}^{d})$. By the Skorohod theorem~\cite{B}, one can construct a new probability space and new random variables without changing the distribution such that (here we don't change the notations) $x^{\epsilon}$ almost surely to $x$ in space $C(0,T;\mathbb{R}^{d})$. Next we determine the limit $x$. For any $\phi\in C^{3}(\mathbb{R})$, $h\in\mathbb{R}^{d}$ and any $t>0$, integration by parts yields
\begin{eqnarray}\label{DA}
&&\phi(\langle x^{\epsilon}(s),h\rangle)-\phi(\langle x^{\epsilon}(0),h\rangle)=\int_{0}^{t}\phi'(\langle x^{\epsilon}(s),h\rangle)\langle\dot{x}^{\epsilon}(s),h\rangle ds\nonumber\\
&&=-\frac{\epsilon}{\alpha}[\phi'(\langle x^{\epsilon}(t),h\rangle)\langle\dot{x}^{\epsilon}(t),h\rangle-\phi'(\langle x^{\epsilon}(0),h\rangle)\langle\dot{x}^{\epsilon}(0),h\rangle]\nonumber\\
&&\quad+\frac{\epsilon}{\alpha}\int_{0}^{t}\phi''(\langle x^{\epsilon}(s),h\rangle)|\langle\dot{x}^{\epsilon}(s),h\rangle|^{2}ds-\frac{1}{\alpha}\int_{0}^{t}\phi'(\langle x^{\epsilon}(s),h\rangle)\langle \nabla V(x^{\epsilon}(s),\mu_{s}^{\epsilon}),h\rangle ds\nonumber\\
&&\quad+\frac{1}{\alpha\sqrt{\epsilon}}\int_{0}^{t}\phi'(\langle x^{\epsilon}(s),h\rangle)\langle\tilde{\mathbb{E}}\eta^{\epsilon}(s,x^{\epsilon}(s)),h \rangle ds.
\end{eqnarray}
Next, let's first prove that
\begin{eqnarray}
\mathbb{E}\int_{0}^{t}\sqrt{\epsilon}\|\dot{x}^{\epsilon}(s)\|^{2}ds\leq C_{T}.
\end{eqnarray}
Multiplying $\sqrt{\epsilon}\dot{x}^{\epsilon}(t)$ on both sides of (\ref{equ:main}) and integrating from $0$ to $t$, we have
\begin{eqnarray}
&&\frac{\epsilon^{\frac{3}{2}}}{2}(\|\dot{x}^{\epsilon}(t)\|^{2}-\|\dot{x}^{\epsilon}(0)\|^{2})+\alpha\sqrt{\epsilon}\int_{0}^{t}\|\dot{x}^{\epsilon}(s)\|^{2}ds\nonumber\\
&&=-\sqrt{\epsilon}\int_{0}^{t}\langle \nabla V(x^{\epsilon}(s),\mu_{s}^{\epsilon}),\dot{x}^{\epsilon}(s)\rangle+\int_{0}^{t}\langle\tilde{\mathbb{E}}\eta^{\epsilon}(s,x^{\epsilon}(s)),\dot{x}^{\epsilon}(s) \rangle ds\nonumber
\end{eqnarray}
\begin{eqnarray}
&&=-\sqrt{\epsilon}\int_{0}^{t}\langle \nabla V(x^{\epsilon}(s),\mu_{s}^{\epsilon}),\dot{x}^{\epsilon}(s)\rangle ds+\langle\tilde{\mathbb{E}}\eta^{\epsilon}(t,x^{\epsilon}(t)),x^{\epsilon}(t) \rangle\nonumber\\
&&-\langle\tilde{\mathbb{E}}\eta^{\epsilon}(0,x^{\epsilon}(0)),x^{\epsilon}(0) \rangle-\int_{0}^{t}\langle x^{\epsilon}(s),\tilde{\mathbb{E}}\partial_{t}\eta^{\epsilon}(s,x^{\epsilon}(s))\rangle ds
\nonumber\\
&&-\int_{0}^{t}\langle x^{\epsilon}(s),\tilde{\mathbb{E}}\nabla_{x}\eta^{\epsilon}(s,x^{\epsilon}(s))y_{0}\rangle e^{-\frac{\alpha}{\epsilon}s} ds\\
&&+\frac{1}{\epsilon}\int_{0}^{t}\int_{0}^{s}e^{-\frac{\alpha}{\epsilon}(s-u)}\langle x^{\epsilon}(s),\tilde{\mathbb{E}}\nabla_{x}\eta^{\epsilon}(s,x^{\epsilon}(s))\nabla V(x^{\epsilon}(u),\mu_{u}^{\epsilon})\rangle duds\nonumber\\
&&-\frac{1}{\epsilon\sqrt{\epsilon}}\int_{0}^{t}\int_{0}^{s}e^{-\frac{\alpha}{\epsilon}(s-u)}\langle x^{\epsilon}(s),\tilde{\mathbb{E}}\nabla_{x}\eta^{\epsilon}(s,x^{\epsilon}(s))\tilde{\mathbb{E}}\eta^{\epsilon}(u,x^{\epsilon}(u))\rangle duds.\nonumber
\end{eqnarray}
By a basic inequality: $2xy\leq x^{2}+y^{2}$, ($\mathbf{H}_{1}$)-($\mathbf{H}_{3}$), (\ref{equ:x2}), Lemma \ref{MoEs}, and a mixing result~{\rm\cite[lemma 5.14]{DW}},
\begin{eqnarray}\label{equ:1.63}
&&\alpha\sqrt{\epsilon}\int_{0}^{t}\mathbb{E}\|\dot{x}^{\epsilon}(s)\|^{2}ds\nonumber\\
&&\leq\frac{1}{2}\int_{0}^{t}\mathbb{E}\|\nabla V(x^{\epsilon}(s),\mu_{s}^{\epsilon})\|^{2}ds+\frac{1}{2}\int_{0}^{t}\mathbb{E}\|\sqrt{\epsilon}\dot{x}^{\epsilon}(s)\|^{2}ds+\frac{1}{2}\mathbb{E}\|\tilde{\mathbb{E}}\eta^{\epsilon}(s,x^{\epsilon}(s))\|^{2}\nonumber\\
&&+\frac{1}{2}\mathbb{E}\|x^{\epsilon}(t)\|^{2}+\frac{1}{2}\mathbb{E}\|\tilde{\mathbb{E}}\eta^{\epsilon}(0,x^{\epsilon}(0))\|^{2}+\frac{1}{2}\|x_{0}\|^{2}+\frac{1}{2}\int_{0}^{t}\mathbb{E}\|x^{\epsilon}(s)\|^{2}ds\nonumber\\
&&+\frac{1}{2}\mathbb{E}\int_{0}^{t}\|\tilde{\mathbb{E}}\partial_{t}\eta^{\epsilon}(s,x^{\epsilon}(s))\|^{2}ds+\|y_{0}\|\int_{0}^{t}\mathbb{E}\|x^{\epsilon}(s)\|^{2}ds+\|y_{0}\|\int_{0}^{t}\mathbb{E}\|\tilde{\mathbb{E}}\nabla_{x}\eta^{\epsilon}(s,x^{\epsilon}(s))\|^{2}ds\nonumber\\
&&+\frac{M_{\eta}}{\epsilon}\int_{0}^{t}\int_{0}^{s}e^{-\frac{\alpha}{\epsilon}(s-u)}\mathbb{E}\Big[\frac{1}{2}\|x^{\epsilon}(s)\|^{2}+\frac{1}{2}\|\nabla V(x^{\epsilon}(u),\mu_{u}^{\epsilon})\|^{2}\Big]duds\nonumber\\
&&+\frac{C}{\epsilon\sqrt{\epsilon}}\int_{0}^{t}\int_{0}^{s}e^{-\frac{\alpha}{\epsilon}(s-u)}(\mathbb{E}\|x^{\epsilon}(s)\tilde{\mathbb{E}}\nabla_{x}\eta^{\epsilon}(s,x^{\epsilon}(s))\|^{2})^{\frac{1}{2}}\nonumber\\
&&\quad\cdot(\mathbb{E}\|\tilde{\mathbb{E}}\eta^{\epsilon}(u,x^{\epsilon}(u))\|^{4})^{\frac{1}{4}}m\Big(\frac{|s-u|}{\epsilon}\Big)^{\frac{1}{4}}duds.\nonumber\\
&&\leq C_{T}.
\end{eqnarray}
that is
\begin{eqnarray}
\alpha\sqrt{\epsilon}\int_{0}^{t}\mathbb{E}\|\dot{x}^{\epsilon}(s)\|^{2}ds\leq C_{T}.
\end{eqnarray}
Note that $\frac{1}{\alpha\sqrt{\epsilon}}\int_{0}^{t}\phi'(\langle x^{\epsilon}(s),h\rangle)\langle\tilde{\mathbb{E}}\eta^{\epsilon}(s,x^{\epsilon}(s)),h \rangle ds$ is also singular term, we need construct a martingale to deal with it. Define
\begin{eqnarray}
Y_{1}^{\epsilon}(t)&=&\frac{1}{\alpha\sqrt{\epsilon}}\int_{t}^{\infty}\mathbb{E}\Big[\phi'(\langle x^{\epsilon}(t),h\rangle)\langle\tilde{\mathbb{E}}\eta^{\epsilon}(s,x^{\epsilon}(t)),h \rangle\Big|\mathcal{F}_{0}^{\frac{t}{\epsilon}}\Big]ds\nonumber\\
&=&\frac{1}{\alpha\sqrt{\epsilon}}\phi'(\langle x^{\epsilon}(t),h\rangle)\langle\tilde{\mathbb{E}}\eta^{\epsilon}(s,x^{\epsilon}(t)),h \rangle \int_{t}^{\infty}m\Big(\frac{s-t}{\epsilon}\Big)ds\nonumber\\
&=&\frac{\sqrt{\epsilon}K}{\alpha}\phi'(\langle x^{\epsilon}(t),h\rangle)\langle\tilde{\mathbb{E}}\eta^{\epsilon}(t,x^{\epsilon}(t)),h \rangle,
\end{eqnarray} 
and 
\begin{eqnarray}\label{Z1}
&&Z_{1}^{\epsilon}(t)=\lim_{\delta\to0}\frac{\mathbb{E}[Y_{1}^{\epsilon}(t+\delta)-Y_{1}^{\epsilon}(t)]}{\delta}\nonumber\\
&=&\frac{\sqrt{\epsilon}K}{\alpha}\lim_{\delta\to0}\frac{\mathbb{E}[(\phi'(\langle x^{\epsilon}(t+\delta),h\rangle)-\phi'(\langle x^{\epsilon}(t),h\rangle))\langle\tilde{\mathbb{E}}\eta^{\epsilon}(t+\delta,x^{\epsilon}(t+\delta)),h \rangle|\mathcal{F}_{0}^{\frac{t}{\epsilon}}]}{\delta}\nonumber\\
&&+\frac{\sqrt{\epsilon}K}{\alpha}\lim_{\delta\to0}\frac{\mathbb{E}[\phi'(\langle x^{\epsilon}(t),h\rangle)(\langle\tilde{\mathbb{E}}\eta^{\epsilon}(t+\delta,x^{\epsilon}(t+\delta)),h \rangle-\langle\tilde{\mathbb{E}}\eta^{\epsilon}(t,x^{\epsilon}(t+\delta)),h \rangle)|\mathcal{F}_{0}^{\frac{t}{\epsilon}}]}{\delta}\nonumber\\
&&+\frac{\sqrt{\epsilon}K}{\alpha}\lim_{\delta\to0}\frac{\mathbb{E}[\phi'(\langle x^{\epsilon}(t),h\rangle)(\langle\tilde{\mathbb{E}}\eta^{\epsilon}(t,x^{\epsilon}(t+\delta)),h \rangle-\langle\tilde{\mathbb{E}}\eta^{\epsilon}(t,x^{\epsilon}(t)),h \rangle)|\mathcal{F}_{0}^{\frac{t}{\epsilon}}]}{\delta}\nonumber\\
&=&\frac{\sqrt{\epsilon}K}{\alpha}\mathbb{E}\Big[\lim_{\delta\to0}\frac{\phi''(\langle x^{\epsilon}(t)+\theta(x^{\epsilon}(t+\delta)-x^{\epsilon}(t)),h\rangle)\langle x^{\epsilon}(t+\delta)-x^{\epsilon}(t),h\rangle}{\delta}\cdot\nonumber\\
&&\quad\quad\quad\quad\quad\quad\quad\quad\quad\quad\quad\quad\quad\quad\quad\quad\quad\langle\tilde{\mathbb{E}}\eta^{\epsilon}(t+\delta,x^{\epsilon}(t+\delta)),h\rangle\Big|\mathcal{F}_{0}^{\frac{t}{\epsilon}}\Big]\nonumber\\
&&+\frac{\sqrt{\epsilon}K}{\alpha}\mathbb{E}\Big[\lim_{\delta\to0}\frac{\phi'(\langle x^{\epsilon}(t),h\rangle)\langle\tilde{\mathbb{E}}\eta^{\epsilon}(t+\delta,x^{\epsilon}(t+\delta))-\tilde{\mathbb{E}}\eta^{\epsilon}(t,x^{\epsilon}(t)),h\rangle}{\delta}\Big|\mathcal{F}_{0}^{\frac{t}{\epsilon}}\Big]\nonumber\\
&&+\frac{\sqrt{\epsilon}K}{\alpha}\mathbb{E}\Big[\lim_{\delta\to0}\frac{\phi'(\langle x^{\epsilon}(t),h\rangle)\langle\tilde{\mathbb{E}}\nabla_{x}\eta^{\epsilon}(t,x^{\epsilon}(t)+\theta(x^{\epsilon}(t+\delta)-x^{\epsilon}(t))(x^{\epsilon}(t+\delta)-x^{\epsilon}(t)),h\rangle}{\delta}\Big|\mathcal{F}_{0}^{\frac{t}{\epsilon}}\Big]\nonumber\\
&=&\frac{\sqrt{\epsilon}K}{\alpha}\phi''(\langle x^{\epsilon}(t),h\rangle)\langle\dot{x}^{\epsilon}(t),h\rangle\langle\tilde{\mathbb{E}}\eta^{\epsilon}(t,x^{\epsilon}(t)),h\rangle-\frac{K\beta}{\alpha\sqrt{\epsilon}}\phi'(\langle x^{\epsilon}(t),h\rangle)\langle\tilde{\mathbb{E}}\eta^{\epsilon}(t,x^{\epsilon}(t)),h\rangle\nonumber\\
&&+\frac{\sqrt{\epsilon}K}{\alpha}\phi'(\langle x^{\epsilon}(t),h\rangle)\langle\tilde{\mathbb{E}}\nabla_{x}\eta^{\epsilon}(t,x^{\epsilon}(t))\dot{x}^{\epsilon}(t),h\rangle,
\end{eqnarray}
Then 
\begin{eqnarray*}
Y_{1}^{\epsilon}(t)-\int_{0}^{t}Z_{1}^{\epsilon}(s)ds
\end{eqnarray*}
is $\mathcal{F}_{0}^{\frac{t}{\epsilon}}$-martingale, that is 
\begin{eqnarray}\label{MY1}
&&M_{1}^{\epsilon}(t)\triangleq Y_{1}^{\epsilon}(t)-\int_{0}^{t}Z_{1}^{\epsilon}(s)ds\nonumber\\
&=&\frac{\sqrt{\epsilon}K}{\alpha}\phi'(\langle x^{\epsilon}(t),h\rangle)\langle\tilde{\mathbb{E}}\eta^{\epsilon}(t,x^{\epsilon}(t)),h \rangle-\frac{\sqrt{\epsilon}K}{\alpha}\int_{0}^{t}\phi''(\langle x^{\epsilon}(s),h\rangle)\langle\dot{x}^{\epsilon}(s),h\rangle\langle\tilde{\mathbb{E}}\eta^{\epsilon}(s,x^{\epsilon}(s)),h\rangle ds\nonumber\\
&+&\frac{K\beta}{\alpha\sqrt{\epsilon}}\int_{0}^{t}\phi'(\langle x^{\epsilon}(s),h\rangle)\langle\tilde{\mathbb{E}}\eta^{\epsilon}(s,x^{\epsilon}(s)),h\rangle ds-\frac{\sqrt{\epsilon}K}{\alpha}\int_{0}^{t}\phi'(\langle x^{\epsilon}(s),h\rangle)\langle\tilde{\mathbb{E}}\nabla_{x}\eta^{\epsilon}(s,x^{\epsilon}(s))\dot{x}^{\epsilon}(s),h\rangle ds\nonumber\\
&=&\frac{\sqrt{\epsilon}K}{\alpha}\phi'(\langle x^{\epsilon}(t),h\rangle)\langle\tilde{\mathbb{E}}\eta^{\epsilon}(t,x^{\epsilon}(t)),h \rangle+\frac{\epsilon^{\frac{3}{2}}K}{\alpha^{2}}[\phi''(\langle x^{\epsilon}(t),h\rangle)\langle\tilde{\mathbb{E}}\eta^{\epsilon}(t,x^{\epsilon}(t)),h \rangle\langle\dot{x}^{\epsilon}(t),h\rangle\nonumber\\
&&\quad-\phi''(\langle x^{\epsilon}(0),h\rangle)\langle\tilde{\mathbb{E}}\eta^{\epsilon}(0,x^{\epsilon}(0)),h \rangle\langle\dot{x}^{\epsilon}(0),h\rangle]\nonumber\\
&&-\frac{\epsilon^{\frac{3}{2}}K}{\alpha^{2}}\int_{0}^{t}\phi'''(\langle x^{\epsilon}(s),h\rangle)|\langle\dot{x}^{\epsilon}(s),h\rangle|^{2}\langle\tilde{\mathbb{E}}\eta^{\epsilon}(s,x^{\epsilon}(s)),h \rangle ds\nonumber\\
&&-\frac{\epsilon^{\frac{3}{2}}K}{\alpha^{2}}\int_{0}^{t}\phi''(\langle x^{\epsilon}(s),h\rangle)\langle\dot{x}^{\epsilon}(s),h\rangle\langle\tilde{\mathbb{E}}\partial_{t}\eta^{\epsilon}(s,x^{\epsilon}(s)),h \rangle ds\nonumber
\end{eqnarray}
\begin{eqnarray}
&&-\frac{\epsilon^{\frac{3}{2}}K}{\alpha^{2}}\int_{0}^{t}\phi''(\langle x^{\epsilon}(s),h\rangle)\langle\dot{x}^{\epsilon}(s),h\rangle\langle\tilde{\mathbb{E}}\nabla_{x}\eta^{\epsilon}(s,x^{\epsilon}(s))\dot{x}^{\epsilon}(s),h \rangle ds\nonumber\\
&&+\frac{\sqrt{\epsilon}K}{\alpha^{2}}\int_{0}^{t}\phi''(\langle x^{\epsilon}(s),h\rangle)\langle\nabla V(x^{\epsilon}(s),\mu_{s}^{\epsilon}),h\rangle\langle\tilde{\mathbb{E}}\eta^{\epsilon}(s,x^{\epsilon}(s)),h \rangle ds\nonumber\\
&&-\frac{K}{\alpha^{2}}\int_{0}^{t}\phi''(\langle x^{\epsilon}(s),h\rangle)|\langle\tilde{\mathbb{E}}\eta^{\epsilon}(s,x^{\epsilon}(s)),h \rangle|^{2}ds+\frac{K\beta}{\alpha\sqrt{\epsilon}}\int_{0}^{t}\phi'(\langle x^{\epsilon}(s),h\rangle)\langle\tilde{\mathbb{E}}\eta^{\epsilon}(s,x^{\epsilon}(s)),h\rangle ds\nonumber\\
&&+\frac{\epsilon^{\frac{3}{2}}K}{\alpha^{2}}[\phi'(\langle x^{\epsilon}(t),h\rangle)\langle\tilde{\mathbb{E}}\nabla_{x}\eta^{\epsilon}(t,x^{\epsilon}(t))\dot{x}^{\epsilon}(t),h\rangle-\phi'(\langle x^{\epsilon}(0),h\rangle)\langle\tilde{\mathbb{E}}\nabla_{x}\eta^{\epsilon}(0,x^{\epsilon}(0))\dot{x}^{\epsilon}(0),h\rangle]\nonumber\\
&&-\frac{\epsilon^{\frac{3}{2}}K}{\alpha^{2}}\int_{0}^{t}\phi''(\langle x^{\epsilon}(s),h\rangle)\langle\dot{x}^{\epsilon}(s),h\rangle\langle\tilde{\mathbb{E}}\nabla_{x}\eta^{\epsilon}(s,x^{\epsilon}(s))\dot{x}^{\epsilon}(s),h\rangle ds\nonumber\\
&&-\frac{\epsilon^{\frac{3}{2}}K}{\alpha^{2}}\int_{0}^{t}\phi'(\langle x^{\epsilon}(s),h\rangle)\langle\tilde{\mathbb{E}}\partial_{t}\nabla_{x}\eta^{\epsilon}(s,x^{\epsilon}(s))\dot{x}^{\epsilon}(s),h\rangle ds\nonumber\\
&&-\frac{\epsilon^{\frac{3}{2}}K}{\alpha^{2}}\int_{0}^{t}\phi'(\langle x^{\epsilon}(s),h\rangle)\langle\tilde{\mathbb{E}}\nabla_{x}^{2}\eta^{\epsilon}(s,x^{\epsilon}(s))\cdot(\dot{x}^{\epsilon}(s),\dot{x}^{\epsilon}(s)),h\rangle ds\nonumber\\
&=&\frac{\sqrt{\epsilon}K}{\alpha}\phi'(\langle x^{\epsilon}(t),h\rangle)\langle\tilde{\mathbb{E}}\eta^{\epsilon}(t,x^{\epsilon}(t)),h \rangle+\frac{\epsilon^{\frac{3}{2}}K}{\alpha^{2}}[\phi''(\langle x^{\epsilon}(t),h\rangle)\langle\tilde{\mathbb{E}}\eta^{\epsilon}(t,x^{\epsilon}(t)),h \rangle\langle\dot{x}^{\epsilon}(t),h\rangle\nonumber\\
&&\quad-\phi''(\langle x^{\epsilon}(0),h\rangle)\langle\tilde{\mathbb{E}}\eta^{\epsilon}(0,x^{\epsilon}(0)),h \rangle\langle\dot{x}^{\epsilon}(0),h\rangle]\nonumber\\
&&-\frac{\epsilon^{\frac{3}{2}}K}{\alpha^{2}}\int_{0}^{t}\phi'''(\langle x^{\epsilon}(s),h\rangle)|\langle\dot{x}^{\epsilon}(s),h\rangle|^{2}\langle\tilde{\mathbb{E}}\eta^{\epsilon}(s,x^{\epsilon}(s)),h \rangle ds\nonumber\\
&&-\frac{\epsilon^{\frac{3}{2}}K}{\alpha^{2}}\int_{0}^{t}\phi''(\langle x^{\epsilon}(s),h\rangle)\langle\dot{x}^{\epsilon}(s),h\rangle\langle\tilde{\mathbb{E}}\partial_{t}\eta^{\epsilon}(s,x^{\epsilon}(s)),h \rangle ds\nonumber\\
&&-\frac{\epsilon^{\frac{3}{2}}K}{\alpha^{2}}\int_{0}^{t}\phi''(\langle x^{\epsilon}(s),h\rangle)\langle\dot{x}^{\epsilon}(s),h\rangle\langle\tilde{\mathbb{E}}\nabla_{x}\eta^{\epsilon}(s,x^{\epsilon}(s))\dot{x}^{\epsilon}(s),h \rangle ds\nonumber\\
&&+\frac{\sqrt{\epsilon}K}{\alpha^{2}}\int_{0}^{t}\phi''(\langle x^{\epsilon}(s),h\rangle)\langle\nabla V(x^{\epsilon}(s),\mu_{s}^{\epsilon}),h\rangle\langle\tilde{\mathbb{E}}\eta^{\epsilon}(s,x^{\epsilon}(s)),h \rangle ds\nonumber\\
&&-\frac{K}{\alpha^{2}}\int_{0}^{t}\phi''(\langle x^{\epsilon}(s),h\rangle)|\langle\tilde{\mathbb{E}}\eta^{\epsilon}(s,x^{\epsilon}(s)),h \rangle|^{2}ds+K\beta[\phi(\langle x^{\epsilon}(t),h\rangle)-\phi(\langle x^{\epsilon}(0),h\rangle)]\nonumber\\
&&+\frac{K\beta\epsilon}{\alpha}[(\phi'(\langle x^{\epsilon}(t),h\rangle)\langle\dot{x}^{\epsilon}(t),h\rangle-\phi'(\langle x^{\epsilon}(0),h\rangle)\langle\dot{x}^{\epsilon}(0),h\rangle)]\nonumber\\
&&-\frac{K\beta\epsilon}{\alpha}\int_{0}^{t}\phi''(\langle x^{\epsilon}(s),h\rangle)|\langle\dot{x}^{\epsilon}(s),h\rangle|^{2}ds+\frac{K\beta}{\alpha}\int_{0}^{t}\phi'(\langle x^{\epsilon}(s),h\rangle)\langle\nabla V(x^{\epsilon}(s),\mu_{s}^{\epsilon}),h\rangle ds\nonumber\\
&&+\frac{\epsilon^{\frac{3}{2}}K}{\alpha^{2}}[\phi'(\langle x^{\epsilon}(t),h\rangle)\langle \tilde{\mathbb{E}}\nabla_{x}\eta^{\epsilon}(t,x^{\epsilon}(t))\dot{x}^{\epsilon}(t),h\rangle-\phi'(\langle x^{\epsilon}(0),h\rangle)\langle \tilde{\mathbb{E}}\nabla_{x}\eta^{\epsilon}(0,x^{\epsilon}(0))\dot{x}^{\epsilon}(0),h\rangle]\nonumber\\
&&-\frac{\epsilon^{\frac{3}{2}}K}{\alpha^{2}}\int_{0}^{t}\phi''(\langle x^{\epsilon}(s),h\rangle)\langle\dot{x}^{\epsilon}(s),h\rangle\langle\tilde{\mathbb{E}}\nabla_{x}\eta^{\epsilon}(s,x^{\epsilon}(s))\dot{x}^{\epsilon}(s),h\rangle ds\nonumber\\
&&-\frac{\epsilon^{\frac{3}{2}}K}{\alpha^{2}}\int_{0}^{t}\phi'(\langle x^{\epsilon}(s),h\rangle)\langle\tilde{\mathbb{E}}\partial_{t}\nabla_{x}\eta^{\epsilon}(s,x^{\epsilon}(s))\dot{x}^{\epsilon}(s),h \rangle ds\nonumber\\
&&-\frac{\epsilon^{\frac{3}{2}}K}{\alpha^{2}}\int_{0}^{t}\phi'(\langle x^{\epsilon}(s),h\rangle)\langle\tilde{\mathbb{E}}\nabla_{x}^{2}\eta^{\epsilon}(s,x^{\epsilon}(s))(\dot{x}^{\epsilon}(s),\dot{x}^{\epsilon}(s)),h\rangle ds,
\end{eqnarray}
where we have used (\ref{equ:main}) and integration by parts in the second equality and (\ref{DA}) is used in the last equality.

Note that
$\frac{K}{\alpha^{2}}\int_{0}^{t}\phi''(\langle x^{\epsilon}(s),h\rangle)|\langle\tilde{\mathbb{E}}\eta^{\epsilon}(s,x^{\epsilon}(s)),h \rangle|^{2}ds$ is still a singular term. For this, let $\Sigma=\mathbb{E}(\tilde{\mathbb{E}}\eta^{\epsilon}(t,x^{\epsilon}(t))\otimes\tilde{\mathbb{E}}\eta^{\epsilon}(t,x^{\epsilon}(t)))$ and define
\begin{eqnarray}
Y_{2}^{\epsilon}(t)&=&\frac{K}{\alpha^{2}}\int_{t}^{\infty}\phi''(\langle x^{\epsilon}(t),h\rangle)\mathbb{E}\Big[|\langle\tilde{\mathbb{E}}\eta^{\epsilon}(s,x^{\epsilon}(t)),h \rangle|^{2}-\langle\Sigma h,h\rangle\Big|\mathcal{F}_{0}^{\frac{t}{\epsilon}}\Big]ds\nonumber\\
&=&\frac{K^{2}\epsilon}{\alpha^{2}}\phi''(\langle x^{\epsilon}(t),h\rangle)[|\langle\tilde{\mathbb{E}}\eta^{\epsilon}(s,x^{\epsilon}(t)),h \rangle|^{2}-\langle\Sigma h,h\rangle],
\end{eqnarray}
and
\begin{eqnarray}
Z_{2}^{\epsilon}(t)&=&\lim_{\delta\to0}\frac{\mathbb{E}[Y_{2}^{\epsilon}(t+\delta)-Y_{2}^{\epsilon}(t)|\mathcal{F}_{0}^{\frac{t}{\epsilon}}]}{\delta}\nonumber\\
&=&\frac{K^{2}\epsilon}{\alpha^{2}}\lim_{\delta\to0}\frac{\mathbb{E}[\phi''(\langle x^{\epsilon}(t+\delta),h\rangle)(|\langle\tilde{\mathbb{E}}\eta^{\epsilon}(t,x^{\epsilon}(t)),h \rangle|^{2}-\langle\Sigma h,h\rangle)}{\delta}\nonumber\\
&&\quad-\frac{\phi''(\langle x^{\epsilon}(t),h\rangle)(|\langle\tilde{\mathbb{E}}\eta^{\epsilon}(t+\delta,x^{\epsilon}(t+\delta)),h \rangle|^{2}-\langle\Sigma h,h\rangle)|\mathcal{F}_{0}^{\frac{t}{\epsilon}}]}{\delta}\nonumber\\
&&+\frac{K^{2}\epsilon}{\alpha^{2}}\lim_{\delta\to0}\frac{\mathbb{E}[\phi''(\langle x^{\epsilon}(t),h\rangle)(|\langle\tilde{\mathbb{E}}\eta^{\epsilon}(t+\delta,x^{\epsilon}(t+\delta)),h \rangle|^{2}-\langle\Sigma h,h\rangle)}{\delta}\nonumber\\
&&\quad-\frac{\phi''(\langle x^{\epsilon}(t),h\rangle)(|\langle\tilde{\mathbb{E}}\eta^{\epsilon}(t,x^{\epsilon}(t+\delta)),h \rangle|^{2}-\langle\Sigma h,h\rangle)|\mathcal{F}_{0}^{\frac{t}{\epsilon}}]}{\delta}\nonumber\\
&&+\frac{K^{2}\epsilon}{\alpha^{2}}\lim_{\delta\to0}\frac{\mathbb{E}[\phi''(\langle x^{\epsilon}(t),h\rangle)(|\langle\tilde{\mathbb{E}}\eta^{\epsilon}(t,x^{\epsilon}(t+\delta)),h \rangle|^{2}-\langle\Sigma h,h\rangle)}{\delta}\nonumber
\end{eqnarray} 
\begin{eqnarray}
&&\quad-\frac{\phi''(\langle x^{\epsilon}(t),h\rangle)(|\langle\tilde{\mathbb{E}}\eta^{\epsilon}(t,x^{\epsilon}(t)),h \rangle|^{2}-\langle\Sigma h,h\rangle)|\mathcal{F}_{0}^{\frac{t}{\epsilon}}]}{\delta}\nonumber\\
&=&\frac{K^{2}\epsilon}{\alpha^{2}}\phi'''(\langle x^{\epsilon}(t),h\rangle)\langle \dot{x}^{\epsilon}(t),h \rangle(|\langle\tilde{\mathbb{E}}\eta^{\epsilon}(t,x^{\epsilon}(t)),h \rangle|^{2}-\langle\Sigma h,h\rangle)\nonumber\\
&&-\frac{K^{2}\beta}{\alpha^{2}}\phi''(\langle x^{\epsilon}(t),h\rangle)(|\langle\tilde{\mathbb{E}}\eta^{\epsilon}(t,x^{\epsilon}(t)),h \rangle|^{2}-\langle\Sigma h,h\rangle)\nonumber\\
&&+\frac{2K^{2}\epsilon}{\alpha^{2}}\phi''(\langle x^{\epsilon}(t),h\rangle)\langle\tilde{\mathbb{E}}\eta^{\epsilon}(t,x^{\epsilon}(t)),h \rangle\langle\tilde{\mathbb{E}}\nabla_{x}\eta^{\epsilon}(t,x^{\epsilon}(t))\dot{x}^{\epsilon}(t),h \rangle,\nonumber
\end{eqnarray}
then $Y_{2}^{\epsilon}(t)-\int_{0}^{t}Z_{2}^{\epsilon}(s)ds$ is a $\mathcal{F}_{0}^{\frac{t}{\epsilon}}-$martingale. That is 
\begin{eqnarray}
M_{2}^{\epsilon}(t)&\triangleq&Y_{2}^{\epsilon}(t)-\int_{0}^{t}Z_{2}^{\epsilon}(s)ds\nonumber\\
&=&\frac{K^{2}\epsilon}{\alpha^{2}}\phi''(\langle x^{\epsilon}(s),h\rangle)[|\langle\tilde{\mathbb{E}}\eta^{\epsilon}(t,x^{\epsilon}(t)),h \rangle|^{2}-\langle\Sigma h,h\rangle]\nonumber\\
&&-\frac{K^{2}\epsilon}{\alpha^{2}}\int_{0}^{t}\phi'''(\langle x^{\epsilon}(s),h\rangle)\langle \dot{x}^{\epsilon}(s),h \rangle(|\langle\tilde{\mathbb{E}}\eta^{\epsilon}(s,x^{\epsilon}(s)),h \rangle|^{2}-\langle\Sigma h,h\rangle)ds\nonumber\\
&&+\frac{K^{2}\beta}{\alpha^{2}}\int_{0}^{t}\phi''(\langle x^{\epsilon}(s),h\rangle)(|\langle\tilde{\mathbb{E}}\eta^{\epsilon}(s,x^{\epsilon}(s)),h \rangle|^{2}-\langle\Sigma h,h\rangle)ds\nonumber\\
&&-\frac{2K^{2}\epsilon}{\alpha^{2}}\int_{0}^{t}\phi''(\langle x^{\epsilon}(s),h\rangle)\langle\tilde{\mathbb{E}}\eta^{\epsilon}(s,x^{\epsilon}(s)),h \rangle\langle\tilde{\mathbb{E}}\nabla_{x}\eta^{\epsilon}(s,x^{\epsilon}(s))\dot{x}^{\epsilon}(s),h \rangle ds,\nonumber
\end{eqnarray}
is a $\mathcal{F}_{0}^{\frac{t}{\epsilon}}$-martingale. Then
\begin{eqnarray}
&&M^{\epsilon}(t)=\frac{1}{K\beta}M_{1}^{\epsilon}(t)+\frac{1}{K^{2}\beta^{2}}M_{2}^{\epsilon}(t)\nonumber\\
&=&\phi(\langle x^{\epsilon}(t),h\rangle)-\phi(\langle x^{\epsilon}(0),h\rangle)+\frac{\sqrt{\epsilon}}{\alpha\beta}\phi'(\langle x^{\epsilon}(t),h\rangle)\langle\tilde{\mathbb{E}}\eta^{\epsilon}(t,x^{\epsilon}(t)),h \rangle\nonumber\\
&&+\frac{\epsilon^{\frac{3}{2}}}{\alpha^{2}\beta}[\phi''(\langle x^{\epsilon}(t),h\rangle)\langle\tilde{\mathbb{E}}\eta^{\epsilon}(t,x^{\epsilon}(t)),h \rangle\langle\dot{x}^{\epsilon}(t),h\rangle\nonumber\\
&&\quad-\phi''(\langle x^{\epsilon}(0),h\rangle)\langle\tilde{\mathbb{E}}\eta^{\epsilon}(0,x^{\epsilon}(0)),h \rangle\langle\dot{x}^{\epsilon}(0),h\rangle]\nonumber\\
&&-\frac{\epsilon^{\frac{3}{2}}}{\alpha^{2}\beta}\int_{0}^{t}\phi'''(\langle x^{\epsilon}(s),h\rangle)|\langle\dot{x}^{\epsilon}(s),h\rangle|^{2}\langle\tilde{\mathbb{E}}\eta^{\epsilon}(s,x^{\epsilon}(s)),h \rangle ds\nonumber\\
&&-\frac{\epsilon^{\frac{3}{2}}}{\alpha^{2}\beta}\int_{0}^{t}\phi''(\langle x^{\epsilon}(s),h\rangle)\langle\dot{x}^{\epsilon}(s),h\rangle\langle\tilde{\mathbb{E}}\partial_{t}\eta^{\epsilon}(s,x^{\epsilon}(s)),h \rangle ds\nonumber\\
&&-\frac{\epsilon^{\frac{3}{2}}}{\alpha^{2}\beta}\int_{0}^{t}\phi''(\langle x^{\epsilon}(s),h\rangle)\langle\dot{x}^{\epsilon}(s),h\rangle\langle\tilde{\mathbb{E}}\nabla_{x}\eta^{\epsilon}(s,x^{\epsilon}(s))\dot{x}^{\epsilon}(s),h \rangle ds\nonumber\\
&&-\frac{\sqrt{\epsilon}}{\alpha^{2}\beta}\int_{0}^{t}\phi''(\langle x^{\epsilon}(s),h\rangle)\langle\nabla V(x^{\epsilon}(s),\mu_{s}^{\epsilon}),h\rangle\langle\tilde{\mathbb{E}}\eta^{\epsilon}(s,x^{\epsilon}(s)),h \rangle ds\nonumber\\
&&-\frac{1}{\alpha^{2}\beta}\int_{0}^{t}\phi''(\langle x^{\epsilon}(s),h\rangle)\langle\Sigma h,h \rangle ds+\frac{\epsilon}{\alpha}[\phi'(\langle x^{\epsilon}(t),h\rangle)\langle\dot{x}^{\epsilon}(t),h\rangle-\phi'(\langle x^{\epsilon}(0),h\rangle)\langle\dot{x}^{\epsilon}(0),h\rangle]\nonumber\\
&&-\frac{\epsilon}{\alpha}\int_{0}^{t}\phi''(\langle x^{\epsilon}(s),h\rangle)|\langle\dot{x}^{\epsilon}(s),h\rangle|^{2}ds+\frac{1}{\alpha}\int_{0}^{t}\phi'(\langle x^{\epsilon}(s),h\rangle)\langle\nabla V(x^{\epsilon}(s),\mu_{s}^{\epsilon}),h\rangle ds\nonumber\\
&&+\frac{\epsilon^{\frac{3}{2}}}{\alpha^{2}\beta}[\phi'(\langle x^{\epsilon}(t),h\rangle)\langle \tilde{\mathbb{E}}\nabla_{x}\eta^{\epsilon}(t,x^{\epsilon}(t))\dot{x}^{\epsilon}(t),h\rangle\nonumber\\
&&\quad-\phi'(\langle x^{\epsilon}(0),h\rangle)\langle \tilde{\mathbb{E}}\nabla_{x}\eta^{\epsilon}(0,x^{\epsilon}(0))\dot{x}^{\epsilon}(0),h\rangle]\nonumber\\
&&-\frac{\epsilon^{\frac{3}{2}}}{\alpha^{2}\beta}\int_{0}^{t}\phi''(\langle x^{\epsilon}(s),h\rangle)\langle\dot{x}^{\epsilon}(s),h\rangle\langle\tilde{\mathbb{E}}\nabla_{x}\eta^{\epsilon}(s,x^{\epsilon}(s))\dot{x}^{\epsilon}(s),h\rangle ds\nonumber\\
&&-\frac{\epsilon^{\frac{3}{2}}}{\alpha^{2}\beta}\int_{0}^{t}\phi'(\langle x^{\epsilon}(s),h\rangle)\langle\tilde{\mathbb{E}}\partial_{t}\nabla_{x}\eta^{\epsilon}(s,x^{\epsilon}(s)),h \rangle ds\nonumber\\
&&-\frac{\epsilon^{\frac{3}{2}}}{\alpha^{2}\beta}\int_{0}^{t}\phi'(\langle x^{\epsilon}(s),h\rangle)\langle\tilde{\mathbb{E}}\nabla_{x}^{2}\eta^{\epsilon}(s,x^{\epsilon}(s))(\dot{x}^{\epsilon}(s),\dot{x}^{\epsilon}(s)),h\rangle ds.\nonumber\\
&&+\frac{\epsilon}{\alpha^{2}\beta^{2}}\phi''(\langle x^{\epsilon}(t),h\rangle)[|\langle\tilde{\mathbb{E}}\eta^{\epsilon}(t,x^{\epsilon}(t)),h \rangle|^{2}-\langle\Sigma h,h\rangle]\nonumber\\
&&-\frac{\epsilon}{\alpha^{2}\beta^{2}}\int_{0}^{t}\phi'''(\langle x^{\epsilon}(s),h\rangle)\langle \dot{x}^{\epsilon}(s),h \rangle(|\langle\tilde{\mathbb{E}}\eta^{\epsilon}(s,x^{\epsilon}(s)),h \rangle|^{2}-\langle\Sigma h,h\rangle)ds\nonumber\\
&&-\frac{2\epsilon}{\alpha^{2}\beta^{2}}\int_{0}^{t}\phi''(\langle x^{\epsilon}(s),h\rangle)\langle\tilde{\mathbb{E}}\eta^{\epsilon}(s,x^{\epsilon}(s)),h \rangle\langle\tilde{\mathbb{E}}\nabla_{x}\eta^{\epsilon}(s,x^{\epsilon}(s))\dot{x}^{\epsilon}(s),h \rangle ds,\nonumber\\
&\triangleq& \phi(\langle x^{\epsilon}(t),h\rangle)-\phi(\langle x^{\epsilon}(0),h\rangle)+\frac{1}{\alpha}\int_{0}^{t}\phi'(\langle x^{\epsilon}(s),h\rangle)\langle\nabla V(x^{\epsilon}(s),\mu_{s}^{\epsilon}),h\rangle ds\nonumber\\
&&-\frac{1}{\alpha^{2}\beta}\int_{0}^{t}\phi''(\langle x^{\epsilon}(s),h\rangle)\langle\Sigma h,h \rangle ds+R^{\epsilon}(t),\nonumber
\end{eqnarray}
where
\begin{eqnarray}
&&\mathbb{E}|R^{\epsilon}(t)|=\mathbb{E}\Big|\frac{\sqrt{\epsilon}}{\alpha\beta}\phi'(\langle x^{\epsilon}(t),h\rangle)\langle\tilde{\mathbb{E}}\eta^{\epsilon}(t,x^{\epsilon}(t)),h \rangle\nonumber\\
&&+\frac{\epsilon^{\frac{3}{2}}}{\alpha^{2}\beta}[\phi''(\langle x^{\epsilon}(t),h\rangle)\langle\tilde{\mathbb{E}}\eta^{\epsilon}(t,x^{\epsilon}(t)),h \rangle\langle\dot{x}^{\epsilon}(t),h\rangle-\phi''(\langle x^{\epsilon}(0),h\rangle)\langle\tilde{\mathbb{E}}\eta^{\epsilon}(0,x^{\epsilon}(0)),h \rangle\langle\dot{x}^{\epsilon}(0),h\rangle]\nonumber\\
&&-\frac{\epsilon^{\frac{3}{2}}}{\alpha^{2}\beta}\int_{0}^{t}\phi'''(\langle x^{\epsilon}(s),h\rangle)|\langle\dot{x}^{\epsilon}(s),h\rangle|^{2}\langle\tilde{\mathbb{E}}\eta^{\epsilon}(s,x^{\epsilon}(s)),h \rangle ds\nonumber\\
&&-\frac{\epsilon^{\frac{3}{2}}}{\alpha^{2}\beta}\int_{0}^{t}\phi''(\langle x^{\epsilon}(s),h\rangle)\langle\dot{x}^{\epsilon}(s),h\rangle\langle\tilde{\mathbb{E}}\partial_{t}\eta^{\epsilon}(s,x^{\epsilon}(s)),h \rangle ds\nonumber\\
&&-\frac{\epsilon^{\frac{3}{2}}}{\alpha^{2}\beta}\int_{0}^{t}\phi''(\langle x^{\epsilon}(s),h\rangle)\langle x^{\epsilon}(s),h\rangle\langle\tilde{\mathbb{E}}\nabla_{x}\eta^{\epsilon}(s,x^{\epsilon}(s))\dot{x}^{\epsilon}(s),h \rangle ds\nonumber\\
&&-\frac{\sqrt{\epsilon}}{\alpha^{2}\beta}\int_{0}^{t}\phi''(\langle x^{\epsilon}(s),h\rangle)\langle\nabla V(x^{\epsilon}(s),\mu_{s}^{\epsilon}),h\rangle\langle\tilde{\mathbb{E}}\eta^{\epsilon}(s,x^{\epsilon}(s)),h \rangle ds\nonumber\\
&&+\frac{\epsilon}{\alpha}[\phi'(\langle x^{\epsilon}(t),h\rangle)\langle\dot{x}^{\epsilon}(t),h\rangle-\phi'(\langle x^{\epsilon}(0),h\rangle)\langle\dot{x}^{\epsilon}(0),h\rangle]\nonumber\\
&&-\frac{\epsilon}{\alpha}\int_{0}^{t}\phi''(\langle x^{\epsilon}(s),h\rangle)|\langle\dot{x}^{\epsilon}(s),h\rangle|^{2}ds+\frac{\epsilon^{\frac{3}{2}}}{\alpha^{2}\beta}[\phi'(\langle x^{\epsilon}(t),h\rangle)\langle \tilde{\mathbb{E}}\nabla_{x}\eta^{\epsilon}(t,x^{\epsilon}(t))\dot{x}^{\epsilon}(t),h\rangle\nonumber\\
&&-\phi'(\langle x^{\epsilon}(0),h\rangle)\langle \tilde{\mathbb{E}}\nabla_{x}\eta^{\epsilon}(0,x^{\epsilon}(0))\dot{x}^{\epsilon}(0),h\rangle]\nonumber\\
&&-\frac{\epsilon^{\frac{3}{2}}}{\alpha^{2}\beta}\int_{0}^{t}\phi''(\langle x^{\epsilon}(s),h\rangle)\langle\dot{x}^{\epsilon}(s),h\rangle\langle\tilde{\mathbb{E}}\nabla_{x}\eta^{\epsilon}(s,x^{\epsilon}(s))\dot{x}^{\epsilon}(s),h\rangle ds\nonumber\\
&&-\frac{\epsilon^{\frac{3}{2}}}{\alpha^{2}\beta}\int_{0}^{t}\phi'(\langle x^{\epsilon}(s),h\rangle)\langle\tilde{\mathbb{E}}\partial_{t}\nabla_{x}\eta^{\epsilon}(s,x^{\epsilon}(s)),h \rangle ds\nonumber\\
&&-\frac{\epsilon^{\frac{3}{2}}}{\alpha^{2}\beta}\int_{0}^{t}\phi'(\langle x^{\epsilon}(s),h\rangle)\langle\tilde{\mathbb{E}}\nabla_{x}^{2}\eta^{\epsilon}(s,x^{\epsilon}(s))(\dot{x}^{\epsilon}(s),\dot{x}^{\epsilon}(s)),h\rangle ds.\nonumber\\
&&+\frac{\epsilon}{\alpha^{2}\beta^{2}}\phi''(\langle x^{\epsilon}(t),h\rangle)[|\langle\tilde{\mathbb{E}}\eta^{\epsilon}(t,x^{\epsilon}(t)),h \rangle|^{2}-\langle\Sigma h,h\rangle]\nonumber\\
&&-\frac{\epsilon}{\alpha^{2}\beta^{2}}\int_{0}^{t}\phi'''(\langle x^{\epsilon}(s),h\rangle)\langle \dot{x}^{\epsilon}(s),h \rangle(|\langle\tilde{\mathbb{E}}\eta^{\epsilon}(s,x^{\epsilon}(s)),h \rangle|^{2}-\langle\Sigma h,h\rangle)ds\nonumber\\
&&-\frac{2\epsilon}{\alpha^{2}\beta^{2}}\int_{0}^{t}\phi''(\langle x^{\epsilon}(s),h\rangle)\langle\tilde{\mathbb{E}}\eta^{\epsilon}(s,x^{\epsilon}(s)),h \rangle\langle\tilde{\mathbb{E}}\nabla_{x}\eta^{\epsilon}(s,x^{\epsilon}(s))\dot{x}^{\epsilon}(s),h \rangle ds\Big|,\nonumber\\
&\leq& \frac{M\sqrt{\epsilon}}{\alpha\beta}\mathbb{E}\|\tilde{\mathbb{E}}\eta^{\epsilon}(t,x^{\epsilon}(t))\|\|h\|+\frac{M\epsilon}{\alpha^{2}\beta}\mathbb{E}|\langle\tilde{\mathbb{E}}\eta^{\epsilon}(t,x^{\epsilon}(t)),h \rangle||\langle\sqrt{\epsilon}\dot{x}^{\epsilon}(t),h\rangle|\nonumber\\
&&+\frac{M\epsilon^{\frac{3}{2}}}{\alpha^{2}\beta}\int_{0}^{t}|\langle\dot{x}^{\epsilon}(s),h\rangle|^{2}|\langle\tilde{\mathbb{E}}\eta^{\epsilon}(s,x^{\epsilon}(s)),h \rangle| ds\nonumber\\
&&+\frac{M\epsilon^{\frac{3}{2}}}{\alpha^{2}\beta}\int_{0}^{t}\mathbb{E}|\langle\dot{x}^{\epsilon}(s),h\rangle||\langle\tilde{\mathbb{E}}\partial_{t}\eta^{\epsilon}(s,x^{\epsilon}(s)),h \rangle| ds\nonumber\\
&&+\frac{M\epsilon^{\frac{3}{2}}}{\alpha^{2}\beta}\int_{0}^{t}\mathbb{E}|\langle x^{\epsilon}(s),h\rangle||\langle\tilde{\mathbb{E}}\nabla_{x}\eta^{\epsilon}(s,x^{\epsilon}(s))\dot{x}^{\epsilon}(s),h \rangle| ds\nonumber
\end{eqnarray}
\begin{eqnarray}
&&+\frac{M\sqrt{\epsilon}}{\alpha^{2}\beta}\int_{0}^{t}|\langle\nabla V(x^{\epsilon}(s),\mu_{s}^{\epsilon}),h\rangle||\langle\tilde{\mathbb{E}}\eta^{\epsilon}(s,x^{\epsilon}(s)),h \rangle| ds\nonumber\\
&&+\frac{M\epsilon}{\alpha}\mathbb{E}|\langle\dot{x}^{\epsilon}(t),h\rangle|+\frac{M\epsilon}{\alpha}|\langle\dot{x}^{\epsilon}(0),h\rangle|+\frac{M\sqrt{\epsilon}}{\alpha}\int_{0}^{t}\sqrt{\epsilon}\mathbb{E}|\langle\dot{x}^{\epsilon}(s),h\rangle|^{2}ds\nonumber\\
&&+\frac{M\epsilon^{\frac{3}{2}}}{\alpha^{2}\beta}[\mathbb{E}|\langle \tilde{\mathbb{E}}\nabla_{x}\eta^{\epsilon}(t,x^{\epsilon}(t))\dot{x}^{\epsilon}(t),h\rangle|+|\langle \tilde{\mathbb{E}}\nabla_{x}\eta^{\epsilon}(0,x^{\epsilon}(0))\dot{x}^{\epsilon}(0),h\rangle|]\nonumber\\
&&+\frac{M\epsilon^{\frac{3}{2}}}{\alpha^{2}\beta}\int_{0}^{t}\mathbb{E}|\langle\dot{x}^{\epsilon}(s),h\rangle||\langle\tilde{\mathbb{E}}\nabla_{x}\eta^{\epsilon}(s,x^{\epsilon}(s))\dot{x}^{\epsilon}(s),h\rangle| ds\nonumber\\
&&+\frac{M\epsilon^{\frac{3}{2}}}{\alpha^{2}\beta}\mathbb{E}|\int_{0}^{t}\langle\tilde{\mathbb{E}}\partial_{t}\nabla_{x}\eta^{\epsilon}(s,x^{\epsilon}(s)),h \rangle| ds\nonumber\\
&&+\frac{M\epsilon^{\frac{3}{2}}}{\alpha^{2}\beta}\int_{0}^{t}\mathbb{E}|\langle\tilde{\mathbb{E}}\nabla_{x}^{2}\eta^{\epsilon}(s,x^{\epsilon}(s))(\dot{x}^{\epsilon}(s),\dot{x}^{\epsilon}(s)),h\rangle| ds+\frac{M\epsilon}{\alpha^{2}\beta^{2}}\mathbb{E}[|\langle\tilde{\mathbb{E}}\eta^{\epsilon}(t,x^{\epsilon}(t)),h \rangle|^{2}+|\langle\Sigma h,h\rangle|]\nonumber\\
&&+\frac{M\epsilon}{\alpha^{2}\beta^{2}}\int_{0}^{t}\mathbb{E}|\langle \dot{x}^{\epsilon}(s),h \rangle|(|\langle\tilde{\mathbb{E}}\eta^{\epsilon}(s,x^{\epsilon}(s)),h \rangle|^{2}+|\langle\Sigma h,h\rangle|)ds\nonumber\\
&&+\frac{2M\epsilon}{\alpha^{2}\beta^{2}}\int_{0}^{t}\mathbb{E}|\langle\tilde{\mathbb{E}}\eta^{\epsilon}(s,x^{\epsilon}(s)),h \rangle||\langle\tilde{\mathbb{E}}\nabla_{x}\eta^{\epsilon}(s,x^{\epsilon}(s))\dot{x}^{\epsilon}(s),h \rangle| ds,\nonumber\\
&\leq&\sqrt{\epsilon}\frac{M\|h\|}{\alpha\beta}M_{\eta}+\epsilon\frac{M\|h\|^{2}}{\alpha^{2}\beta}M_{\eta}(\mathbb{E}\|\sqrt{\epsilon}\dot{x}^{\epsilon}(t)\|^{2})^{\frac{1}{2}}+\sqrt{\epsilon}\frac{M\|h\|^{3}M_{\eta}}{\alpha^{2}\beta}\int_{0}^{t}(\mathbb{E}\|\sqrt{\epsilon}\dot{x}^{\epsilon}(s)\|^{4})^{\frac{1}{2}}ds\nonumber\\
&&+\epsilon\frac{M\|h\|^{2}M_{\eta}}{\alpha^{2}\beta}\int_{0}^{t}(\mathbb{E}\|\sqrt{\epsilon}\dot{x}^{\epsilon}(s)\|^{2})^{\frac{1}{2}}ds+\epsilon\frac{MC\|h\|^{2}}{\alpha^{2}\beta}\int_{0}^{t}(\mathbb{E}\|\sqrt{\epsilon}\dot{x}^{\epsilon}(s)\|^{2})^{\frac{1}{2}}(\mathbb{E}\|x^{\epsilon}(s)\|^{2})^{\frac{1}{2}}ds\nonumber\\
&&+\sqrt{\epsilon}\frac{M\|h\|^{2}}{\alpha^{2}\beta}2C_{V}M_{\eta}\int_{0}^{t}(\mathbb{E}\|x^{\epsilon}(s)\|^{2}+\|\nabla V(0,\delta_{0})\|^{2})^{\frac{1}{2}}ds+\frac{M\sqrt{\epsilon}\|h\|}{\alpha}(\mathbb{E}\|\sqrt{\epsilon}\dot{x}^{\epsilon}(s)\|^{2})^{\frac{1}{2}}\nonumber\\
&&+\frac{M\sqrt{\epsilon}\|h\|}{\alpha}\|x_{0}\|+\frac{M\sqrt{\epsilon}\|h\|^{2}}{\alpha}\int_{0}^{t}\mathbb{E}\|\sqrt{\epsilon}\dot{x}^{\epsilon}(s)\|^{2}ds+\epsilon\frac{M\|h\|}{\alpha^{2}\beta}M_{\eta}(\mathbb{E}\|\sqrt{\epsilon}\dot{x}^{\epsilon}(t)\|^{2})^{\frac{1}{2}}\nonumber\\
&&+\sqrt{\epsilon}\frac{M\|h\|^{2}}{\alpha^{2}\beta}M_{\eta}\int_{0}^{t}\mathbb{E}\|\sqrt{\epsilon}\dot{x}^{\epsilon}(s)\|^{2}ds+\epsilon^{\frac{3}{2}}\frac{M\|h\|}{\alpha^{2}\beta}M_{\eta}T+\sqrt{\epsilon}\frac{M\|h\|}{\alpha^{2}\beta}M_{\eta}\int_{0}^{t}\mathbb{E}\|\sqrt{\epsilon}\dot{x}^{\epsilon}(s)\|^{2}ds\nonumber\\
&&+\epsilon\frac{M\|h\|^{2}}{\alpha^{2}\beta}M_{\eta}^{2}+\frac{M\sqrt{\epsilon}\|h\|^{3}M_{\eta}^{2}}{\alpha^{2}\beta^{2}}\int_{0}^{t}(\mathbb{E}\|\sqrt{\epsilon}\dot{x}^{\epsilon}(s)\|^{2})^{\frac{1}{2}}ds\nonumber\\
&&+\frac{2M\sqrt{\epsilon}\|h\|}{\alpha^{2}\beta^{2}}M_{\eta}\int_{0}^{t}(\mathbb{E}\|\sqrt{\epsilon}\dot{x}^{\epsilon}(s)\|^{2})^{\frac{1}{2}}ds\nonumber\to0,\quad\epsilon\to0,
\end{eqnarray}
where $M$ denotes the boundedness of $\phi$ and its derivatives of any order.
Then $\{M^{\epsilon}(\cdot)\}_{0\leq \epsilon\leq 1}$ converges to $\{M(\cdot)\}$ as $\epsilon\to0$. Where $M(t)$ satisfies
\begin{eqnarray}
M(t)&=&\phi(\langle x^{\epsilon}(t),h\rangle)-\phi(\langle x^{\epsilon}(0),h\rangle)+\frac{1}{\alpha}\int_{0}^{t}\phi'(\langle x^{\epsilon}(s),h\rangle)\langle\nabla V(x^{\epsilon}(s),\mu_{s}^{\epsilon}),h\rangle ds\nonumber\\
&&-\frac{1}{\alpha^{2}\beta}\int_{0}^{t}\phi''(\langle x^{\epsilon}(s),h\rangle)\langle\Sigma h,h \rangle ds,
\end{eqnarray}
and by the representation theorem of martingale~\cite{M}, $\{x(t)\}$ is the unique weak solution of the following stochastic differential equation  
\begin{equation}\label{equ:thm}
dx(t)=-\frac{1}{\alpha}\nabla V(x(t),\mu_{t})dt+\frac{\sqrt{\frac{\Sigma}{\beta}}}{\alpha}dB(t),
\end{equation}
where $\Sigma=\mathbb{E}(\tilde{\mathbb{E}}\eta^{\epsilon}(t,x^{\epsilon}(t))\otimes\tilde{\mathbb{E}}\eta^{\epsilon}(t,x^{\epsilon}(t)))$ and $B(t)$ is a standard Brownian motion.

Then we have the following result
\begin{theorem}
The family $\{x^{\epsilon}(t)\}_{0<\epsilon\leq1}$ converges in distribution to $x(t)$ which satisfies (\ref{equ:thm}).
\end{theorem}
So we complete the proof of Theorem \ref{equ:thm1}.

\end{document}